\documentclass[reqno,12pt,a4paper]{amsart}

\usepackage{amsmath,amsfonts,amstext,amsxtra,amssymb,amsthm,mathtools} 
\usepackage{array}
\usepackage[utf8]{inputenc}
\usepackage[english]{babel}

\usepackage[backend=biber,style=alphabetic,sorting=nyt,giveninits=true,maxbibnames=5]{biblatex}
\usepackage{csquotes}
\DeclareFieldFormat*{title}{\mkbibemph{#1}}

\DeclareFieldFormat*{volume}{\mkbibbold{#1}}
\DeclareFieldFormat*{number}{\mkbibbold{#1}}

\usepackage[left=2.5cm,paperwidth=21cm,right=2.5cm,paperheight=29.7cm,textheight=23.5cm,top=2.5cm]{geometry}

\usepackage{tikz}
\usetikzlibrary{matrix,quotes,cd,decorations.markings}

\usepackage{enumitem} 
\usepackage{hyperref} 
\hypersetup{hidelinks,colorlinks,citecolor={blue!70!black},urlcolor=black,linkcolor={green!70!black}}
\usepackage[noabbrev]{cleveref} 
\usepackage{xfrac} 
\usepackage[page]{appendix} 

\title{On zero-cycles of varieties over Laurent fields}
\author{Jan Lange}
\address{Institute of Algebraic Geometry \\ Leibniz University Hannover \\ Welfengarten 1 \\ 30167 Hannover, Germany}
\email{\href{mailto:lange@math.uni-hannover.de}{lange@math.uni-hannover.de}}

\date{February 22, 2024}

\newtheorem{theorem}{Theorem}[section]
\newtheorem{example}[theorem]{Example}

\newtheorem{proposition}[theorem]{Proposition}
\newtheorem{cor}[theorem]{Corollary}
\newtheorem{lemma}[theorem]{Lemma}
\newtheorem*{theorem*}{Theorem}

\newtheorem{remark}[theorem]{Remark}
\newtheorem{setup}[theorem]{Setup}

\theoremstyle{definition}
\newtheorem{definition}[theorem]{Definition}
\newtheorem{construction}[theorem]{Construction}

\newcolumntype{C}{>{$}c<{$}}

\DeclareMathOperator{\Azero}{A_0}
\DeclareMathOperator{\blowup}{Bl}
\DeclareMathOperator{\Brauer}{Br}
\DeclareMathOperator{\CH}{CH}
\DeclareMathOperator{\charac}{char}
\DeclareMathOperator{\Cone}{Cone}
\DeclareMathOperator{\Homo}{Hom}
\DeclareMathOperator{\Ima}{Im}
\DeclareMathOperator{\mult}{mult}
\DeclareMathOperator{\pr}{pr}
\DeclareMathOperator{\Spec}{Spec}
\DeclareMathOperator{\Wall}{RelWall}

\newcommand{\N}{\mathbb{N}}
\newcommand{\Z}{\mathbb{Z}}
\newcommand{\Q}{\mathbb{Q}}
\newcommand{\R}{\mathbb{R}}
\newcommand{\C}{\mathbb{C}}

\newcommand{\aff}{\mathbb{A}}
\newcommand{\proj}{\mathbb{P}}

\newcommand{\restr}[1]{\left. #1 \right|} 
\newcommand{\innerpro}[2]{\left\langle #1, #2 \right\rangle}

\allowdisplaybreaks

\numberwithin{equation}{section}

\begin{document}

\begin{abstract}
    We generalize a recent result of Pavic--Schreieder regarding the surjectivity of the obstruction morphism defined in \cite{PS23}. As a consequence of this result, we show that geometrically (retract) rational varieties over a Laurent field of characteristic 0, which admit a strictly semi-stable model, have trivial Chow group of zero-cycles. Our key new ingredient comes from toric geometry.
\end{abstract}

\maketitle

\section{Introduction}

Kontsevich--Tschinkel \cite{KT19} and Nicaise--Shinder \cite{NS19} study the behaviour of rationality and stable rationality in families in characteristic 0. In particular for degenerations of a smooth projective variety, they construct a motivic obstruction to (stable) rationality which depends only on the special fibre of the degeneration. This approach was successfully applied by Nicaise--Ottem \cite{NO22} to show the stable irrationality of quartic fivefolds and several complete intersections by reducing to previously known irrationality results \cite{HPT18,Sch19JAMS}. Building on \cite{Sch19JAMS}, \cite{Moe23} uses the approach in \cite{NO22} to improve Schreieder's logarithmic bound for stable irrational hypersurfaces in characteristic $0$. 

Motivated by the cycle-theoretic approaches to stable rationality in \cite{Voi15,CTP16,Sch19Duke}, Pavic--Schreieder \cite{PS23} introduce a Chow-theoretic analogue of the motivic approach for strictly semi-stable schemes $\mathfrak{X}$ over a dvr $R$. These are regular, flat, and proper $R$-schemes whose special fibre is a geometrically reduced simple normal crossing divisor $Y = \bigcup_{i \in I} Y_i$ with irreducible components $Y_i$. For each such $\mathfrak{X} \to \Spec R$ they consider the complex
\begin{equation}\label{eq:complex}
    \bigoplus\limits_{j \in I} \CH_1(Y_j) \overset{\Phi}{\longrightarrow} \bigoplus\limits_{i \in I} \CH_0(Y_i) \xrightarrow[]{\sum_i \deg} \Z \longrightarrow 0,
\end{equation}
where $\Phi = \sum\limits_{i \in I} \sum\limits_{j \in I} \iota_i^\ast (\iota_j)_\ast$ with $\iota_i \colon Y_i \hookrightarrow \mathfrak X$ the natural inclusion  for $i \in I$.

\begin{theorem}\label{thm:main}
    Let $R$ be a discrete valuation ring with algebraically closed residue field $k$, and let $\mathfrak{X} \to \Spec R$ be a strictly semi-stable projective $R$-scheme. Assume that the geometric generic fibre admits a decomposition of the diagonal (e.g. is  retract rational). Then the complex \eqref{eq:complex} is exact after base-change to any field extension $L/k$, i.e. the complex \eqref{eq:complex} is exact for the strictly semi-stable family $\mathfrak{X} \times_{R} A \to \Spec A$, where $A/R$ is any unramified extension of dvr's with induced extension $L/k$ of residue fields.
\end{theorem}

The theorem implies that the complex \eqref{eq:complex} is exact modulo $m$ for all $m \in \N$, as the tensor product is right exact. We would like to emphasize that the assumption on the residue field of $R$ being algebraically closed is crucial, see \Cref{rem:counterexample}. \Cref{thm:main} improves \cite[Theorem 1.2 (2)]{PS23} which shows the exactness of the complex tensored with $\Z/2$ if the special fibre is a chain of Cartier divisors. We can remove the restriction to $\Z/2$-coefficients and allow an arbitrary configuration of the special fibre in the strictly semi-stable family. The theorem of Pavic--Schreieder implies the (retract) irrationality of very general quartic fivefolds \cite[Theorem 1.1]{PS23} and of very general complete intersections of two cubics in $\proj^7$ \cite[Theorem 1.1]{LS23}. One can hope to approach the other complete intersections in \cite{NO22} with \Cref{thm:main}. 

As a consequence of \Cref{thm:main}, we study in this paper the Chow group of zero-cycles of geometrically rational varieties over Laurent fields $k((t))$ where $k$ is algebraically closed. Colliot-Th\'{e}l\`{e}ne shows in \cite[Theorem A (iv)]{CT83} that the Chow group of degree 0 zero-cycles is trivial for geometrically rational surfaces over fields of characteristic $0$ and cohomological dimension at most $1$. In particular \cite{CT83} applies to Laurent fields in characteristic $0$. More generally, Tian considers rationally connected varieties over Laurent fields in characteristic $0$ and shows the triviality of the Chow group of zero-cycles unconditionally in dimension at most $3$ and in all dimensions if the Tate conjecture for surfaces holds, see \cite[Theorem 1.1]{Ti20}. We obtain a partial, but unconditional result in this direction and a similar result in positive characteristic.

\begin{cor}\label{cor:main}
    Let $X$ be a smooth, projective, geometrically rational variety over a Laurent field $k((t))$, which admits a strictly semi-stable projective model $\mathfrak{X} \to \Spec k[[t]]$. Assume that $k$ is algebraically closed. Then the following holds:
    \begin{enumerate}[label=(\roman*)]
        \item The degree map $\deg \colon \CH_0(X) \to \Z$ is an isomorphism if $\charac k = 0$.
        \item If $p = \charac k > 0$, then $\deg \colon \CH_0(X)/l \to \Z/l$ is an isomorphism for every $l$ coprime to $p$.
    \end{enumerate}
\end{cor}

Over $\C((t))$, the result is known more generally for rationally simply connected varieties by \cite[Theorem 1.5]{Pir11}. For smooth families $\mathfrak X \to \Spec k[[t]]$, \Cref{cor:main} (i) follows from \cite[Theorem 2 (2)]{Kol04}.

\begin{remark}
    \Cref{cor:main} holds over fraction fields of excellent, henselian discrete valuation rings with algebraically closed residue field $k$, see \Cref{cor:applicationCH0}.
\end{remark}

We briefly sketch the strategy for the two results: The assumption of \Cref{thm:main} ensures that the generic fibre of $\mathfrak X \to \Spec R$ admits a decomposition of the diagonal after a finite field extension. This implies by \cite[Theorem 1.2 (1)]{PS23} that the complex \eqref{eq:complex} is exact for a suitable resolution $\Tilde{\mathfrak X} \to \Spec \Tilde{R}$ after a finite base change $\Tilde{R}/R$. From this we aim to deduce the exactness over $R$. The issue is that the suitable resolution $\Tilde{\mathfrak{X}} \to \Spec \Tilde{R}$ has many more components in the special fibre, so the associated complexes \eqref{eq:complex} are quite different. We describe the suitable resolution $\Tilde{\mathfrak{X}}$ \'{e}tale locally by toric geometry and obtain a global description via simplicial complexes. This way we can reduce \Cref{thm:main} to a combinatorial problem which we solve in \Cref{prop:key}.

\Cref{cor:main} follows from \Cref{thm:main} via Fulton's localization exact sequence
$$
    \CH_1(Y) \longrightarrow \CH_1(\mathfrak{X}) \longrightarrow \CH_0(X) \longrightarrow 0
$$
by using Saito-Sato's \cite{SS10} bijectivity result for the {\'e}tale cycle class map.

Some preliminaries are given in \Cref{sec:Preliminaries}. In \Cref{sec:toricgeometry}, we briefly recall some toric intersection theory and provide a description of a particular affine toric singularity and its resolution. The main technical aspects of this paper are contained in \Cref{sec:CH1resolution}, where we introduce \emph{refinements of abstract simplicial complexes} and use it to describe the components of the resolution and one-cycles on the special fibre. In \Cref{sec:analysis}, we use the results from \Cref{sec:CH1resolution} to compare the complex \eqref{eq:complex} with the complex of a suitable resolution after a finite base change by constructing two auxiliary functions. This enables us to prove \Cref{thm:main} in \Cref{sec:mainresult}, where we also explain the argument for \Cref{cor:main}. In \Cref{appendix}, we provide two concrete examples, for which we illustrate the key proposition - \Cref{prop:key} - by a direct computation. We achieve this by explicitly spelling out the constructions and arguments from \Cref{sec:CH1resolution,sec:analysis} for each example.

\section*{Acknowledgement}

I am grateful to my advisor Stefan Schreieder for his support and his many helpful suggestions regarding this work. I thank Jean-Louis Colliot-Th\'{e}l\`{e}ne for his comments. The author is supported by the Studienstiftung des deutschen Volkes. This project received partial funding from the European Research Council (ERC) under the European Union’s Horizon 2020 research and innovations programe under grant agreement No. 948066 (ERC - StG RationAlgic).

\section{Preliminaries}\label{sec:Preliminaries}

\subsection{Notations}
    Let $k$ be a field. A \emph{$k$-variety} (or \emph{variety}) $X$ is an integral, separated scheme of finite type over $k$. The base-change of a separated scheme $X$ over a ring $R$ to some ring extension $A/R$ is denoted by $$
        X_A := X \times_R A := X \times_{\Spec R} \Spec A.
    $$
    
    Let $Y$ be a scheme of finite type over a field $k$. The \emph{group of $l$-cycles} of $Y$ is the free abelian group generated by subvarieties of dimension $l$. The \emph{Chow group of $l$-cycles} $\CH_l(Y)$ of $Y$ is the quotient of the group of $l$-cycles by rational equivalence.

    For an abelian group $G$ and an integer $l \in \Z_{\geq 0}$, we denote $G \otimes_\Z \Z/l\Z$ by $G/l$.

\subsection{Decomposition of the diagonal}

A $k$-variety $X$ of dimension $n$ admits a (Chow-theoretic) decomposition of the diagonal, if the diagonal $\Delta_X \subset X \times_k X$ satisfies
$$
    [\Delta_X] = [X \times z] + [Z] \in \CH_n(X \times_k X),
$$
where $z$ is a zero-cycle in $X$ of degree $1$ and $Z$ is an $n$-cycle on $X \times_k X$ which does not dominate the first factor. For example, retract rational varieties admit a decomposition of the diagonal, see e.g. \cite[Lemma 7.5]{Sch21}.

\subsection{Strictly semi-stable families}\label{sec:strictlysemistable}

Let $R$ be a discrete valuation ring. A flat, proper, and regular $R$-scheme $\mathfrak{X} \to \Spec R$ is called \emph{strictly semi-stable}, if the special fibre $\mathfrak{X} \times_R k$ is a geometrically reduced, simple normal crossing divisor, i.e. every component of the special fibre is a smooth Cartier divisor in $\mathfrak{X}$ and the scheme-theoretic intersection of the components is either empty or smooth of the expected codimension.

Let $R$ be a discrete valuation ring and let $R \subset \Tilde{R}$ be an unramified extension of discrete valuation rings, i.e. an extension $R \to \Tilde{R}$ such that $m_{\Tilde{R}} = m_R \Tilde{R}$ where $m_R$ and $m_{\Tilde{R}}$ are the maximal ideal of $R$ and $\Tilde{R}$, respectively. Then for every strictly semi-stable $R$-scheme $\mathfrak{X}$, the base change $\mathfrak{X}_{\Tilde{R}} = \mathfrak{X} \times_R \Tilde{R}$ is a strict semi-stable $\Tilde{R}$-scheme by \cite[Proposition 1.3]{Har01}. For finite ramified extension of dvr's the base-change $\mathfrak{X}_{\Tilde{R}}$ becomes a strictly semi-stable $\Tilde{R}$-scheme after a finite sequence of blow-ups.

\begin{proposition}[{\cite[Proposition 2.2]{Har01}}]\label{prop:Hartlresolution}
    Let $R \subset \Tilde{R}$ be a finite ramified extension of discrete valuation rings. Let $\mathfrak{X}$ be a strictly semi-stable $R$-scheme with special fibre $Y$. Assume that the irreducible components of $Y$ are geometrically integral. Then there exists a finite sequence of blow-ups $\Tilde{\mathfrak{X}} := V^m \to V^{m-1} \to \dots \to V^0 =: \mathfrak{X}_{\Tilde{R}}$ such that $\Tilde{\mathfrak{X}}$ is a strictly semi-stable $\Tilde{R}$-scheme and the center of each blow-up $V_{i+1} \to V_i$ is an irreducible component of the special fibre of $V_i$.
\end{proposition}

\subsection{Basic constructions in toric geometry}

Throughout this section, let $k$ be a field. A (split) toric variety is a normal $k$-variety $X$ which contains a torus $T = (k^\ast)^{\dim X}$ as an open dense subset. Toric varieties are also constructed from collections of cones, called fans, and many properties of the toric variety relate to combinatorical data of the fan. We recall the construction of toric varieties from fans, introduce some notation and state a few standard facts from toric geometry.

Let $N \cong \Z^n$ be a lattice and denote by $M = \Homo_\Z(N,\Z) \cong \Z^n$ its dual lattice. Consider the natural duality pairing $\innerpro{\cdot}{\cdot} \colon M \times N \to \Z$. Morover, we denote the tensor products $M \otimes_\Z \R$ and $N \otimes_\Z \R$ by $M_\R$ and $N_\R$, respectively.

A \emph{(strongly convex, rational, polyhedral or short: scrp) cone} in $N_\R$ is a set $$
    \sigma := \Cone(S) := \left\{\sum\limits_{u \in S} \lambda_u u : \lambda_u \geq 0\right\} \subset N_\R,
$$
where $S \subset N$ is a finite subset such that $\sigma \cap (-\sigma) = \{0\}$. The elements in $S$ are called \emph{generators} of $\sigma$. The \emph{dimension} of $\sigma$ is the dimension of its linear span in $N_\R$. The \emph{dual cone} of $\sigma$ is the subset $$
    \sigma^\vee := \{m \in M_\R : \innerpro{m}{u} \geq 0 \text{ for all } u \in \sigma\} \subset M_\R,
$$
which is again a (rational, polyhedral, not necessarily strongly convex) cone. Gordan's lemma says that $\sigma^\vee \cap M$ is a finitely generated semi-group such that $k[\sigma^\vee \cap M]$ is an integral domain, see e.g. \cite[Proposition 1.2.17]{CLS11}. Then the \emph{affine toric variety} $U_\sigma$ associated to a cone $\sigma$ is $U_\sigma = \Spec k[\sigma^\vee \cap M]$.

A \emph{face} $\tau$ of a (scrp) cone $\sigma$ is a subset of the form $$
    \tau = \{u \in \sigma: \innerpro{m}{u} = 0 \} \subset \sigma
$$
for some $m \in \sigma^\vee$ and it is again a (scrp) cone. A collection of (scrp) cones $\Sigma$ is called a \emph{fan} if each face of a cone is also in $\Sigma$ and if the intersection of two cones in $\Sigma$ is a common face of both. These two conditions ensure that we can glue the affine toric varieties associated to each cone $\sigma \in \Sigma$ along the intersection of the cones to obtain a toric variety, which we denote by $X_\Sigma$. We denote the set of $n$-dimensional cones in the fan $\Sigma$ by $\Sigma(n)$ and we usually call $1$-dimensional cones rays. Recall also the following facts in toric geometry:
\begin{enumerate}[label=(\roman*)]
    \item A $\Z$-linear map of lattices $N \to N'$ is \emph{compatible} with a pair of fans $\Sigma$ in $N_\R$ and $\Sigma'$ in $N'_\R$ if the $\R$-linear extension maps every cone $\sigma \in \Sigma$ into a cone $\sigma' \in \Sigma'$. Such a compatible map of lattices gives rise to a toric morphism of the associated toric varieties, i.e. a morphism of the varieties such that the restriction to the tori is a group homomorphism, see \cite[Theorem 4.1]{Oda78}.
    \item The orbit-cone correspondence yields an inclusion-reversing bijection between cones in a fan and orbits of the torus action of its associated toric variety, see e.g. \cite[Theorem 4.2]{Oda78}. If $\sigma \in \Sigma$ is a cone in the fan, then we denote the corresponding orbit in $X_\Sigma$ by $O(\sigma)$ and its Zariski closure by $V(\sigma)$.
    \item The generators $v_i$ of the $1$-dimensional faces $\rho_i$ of a (scrp) cone $\sigma \subset N_\R$ yield a set of generators for $\sigma$. Up to scaling $v_i$, we can assume that $v_i \in (N \setminus \{0\}) \cap \rho_i$ is of minimal length. Then the $v_i$'s are called the \emph{minimal generators} of $\sigma$.  
    \item We say that a cone is \emph{simplicial} if its minimal generators are $\R$-linear independent. A cone is called \emph{regular} if its minimal generators form part of a $\Z$-basis of the lattice $N$. We say a fan is \emph{simplicial} or \emph{regular}, if all its cones are so. Note that regularity of the fan is equivalent to the regularity of the associated toric variety, see \cite[Theorem 4.3]{Oda78}.
\end{enumerate}

For a more detailed account of the constructions and the facts in toric geometry over arbitrary fields, we refer the reader to \cite{Oda78} and the survey \cite{Dan78}. An earlier description of toric varieties over algebraically closed fields can be found in \cite{KKMS}. We refer the reader to \cite{CLS11} for a modern treatment of toric geometry (over $\C$).

\subsection{The general setup}

We state here the general setup in this paper and refer to it later on.

\begin{setup}\label{setup}
    Let $R$ be a discrete valuation ring with residue field $k$. Let $\Tilde{R}/R$ be a finite extension of discrete valuation rings of ramification index $r$ with induced extension $L/k$ of residue fields. Let $\mathfrak X \to \Spec R$ be a strictly semi-stable $R$-scheme with special fibre $Y = \mathfrak X \times_R k$. Assume that the irreducible components of $Y$ are geometrically integral. Let $\Tilde{\mathfrak X} \to \Spec \Tilde{R}$ be a resolution of $\mathfrak{X}_{\Tilde{R}} \to \Spec \Tilde{R}$ as in \Cref{prop:Hartlresolution} and denote the natural morphism by $q \colon \Tilde{\mathfrak{X}} \to \mathfrak{X}$. Let $\Tilde{Y} = \Tilde{\mathfrak{X}} \times_{\Tilde{R}} L$ be the special fibre of $\Tilde{\mathfrak{X}} \to \Spec \Tilde{R}$ and denote the restriction of $q$ to $\Tilde{Y}$ also by $q \colon \Tilde{Y} \to Y$.
\end{setup}

\section{Resolution of a toric singularity}\label{sec:toricgeometry}

The proof of \Cref{thm:main} relies on the analysis of a certain type of toric singularity together with an iterative description of its resolution as well as a combinatorial description of the intersection of toric curves with toric divisors. We recall the necessary intersection theory from \cite{CLS11} in the first part of this section. The second part consists of a description of the toric singularity together with a resolution process.

\subsection{Some intersection theory}

We recall the required intersection theory which is described by combinatorical data of the fan. Here, we are interested in the intersection of toric divisors with one-cycles, which is described in \cite[Section 6.4]{CLS11}. 

Let $\Sigma$ be a simplicial fan in $N_\R \cong \R^n$. Recall that we denote by $\Sigma(k)$ the set of $k$-dimensional cones in $\Sigma$ and rays are one-dimensional cones $\rho \in \Sigma(1)$. Let $\tau \in \Sigma(n-1)$ be the intersection of two cones $\sigma,\sigma'$ in $\Sigma(n)$. Since $\Sigma$ is simplicial, we can label the minimal generators $u_i$ of $\sigma, \sigma'$, and $\tau$ such that
\begin{equation}\label{eq:cones}
        \sigma = \Cone(u_{1},\dots,u_{n}), \quad
        \sigma' = \Cone(u_{2},\dots,u_{n+1}), \quad
        \tau = \Cone(u_2,\dots,u_{n}).
\end{equation}

Moreover, since the generators $u_{1},\dots,u_{n+1}$ are linearly dependent, they satisfies a (up to multiplication with a constant unique) linear relation, called the \emph{wall relation}
\begin{equation}\label{eq:wallrelation}
    \sum\limits_{i = 1}^{n+1} b_i u_{i} = 0.
\end{equation}
Since $u_2,\dots,u_n$ are linearly independent, we see that $u_1,u_{n+1}$ are both non-zero.

\begin{proposition}[{\cite[Proposition 6.4.4]{CLS11}}]\label{prop:intersectionformula}
    Let $\tau = \sigma \cap \sigma' \in \Sigma(n-1)$ be as in \eqref{eq:cones} and let $V(\tau)$ be the Zariski closure of the orbit corresponding to $\tau$. Moreover, let $\rho_i = \R_{\geq 0} u_i$ be the rays in $N_\R$ generated by $u_i$. Then for every ray $\rho \in \Sigma(1)$ with associated divisor $D_\rho = V(\rho)$ in the toric variety $X_\Sigma$, the intersection number $D_\rho \cdot V(\tau)$ is given by
    $$
        D_\rho \cdot V(\tau) = \begin{cases}
            \frac{\mult(\tau)}{\mult(\sigma)} & \text{if } \rho = \rho_1, \\
            \frac{\mult(\tau)}{\mult(\sigma')} = \frac{b_{n+1} \mult(\tau)}{b_1 \mult(\sigma)} & \text{if } \rho = \rho_{n+1}, \\
            \frac{b_i \mult(\tau)}{b_1 \mult(\sigma)} & \text{if } \rho = \rho_i,  i \neq 1, n+1, \\
            0 & \text{otherwise,}
        \end{cases}
    $$
    where the multiplicity of a cone $\gamma \subset N_\R$ with minimal generators $v_1,\dots,v_k$ is the index of the sublattice $\Z v_1 + \dots + \Z v_k \subset N_\gamma = (\R v_1 + \dots + \R v_k) \cap N$.
\end{proposition}

\begin{proof}[Sketch of the proof]
    Using the uniqueness up to scalar of the wall relation together with the observation that $\sum\limits_{\rho \in \Sigma(1)} \left(D_\rho \cdot V(\tau)\right) u_\rho = 0$ in $N_\R$, where $u_\rho$ is the minimal generator of the ray $\rho \in \Sigma(1)$, it suffices to compute $D_{\rho_1} \cdot V(\tau)$ which can be done explicitly, see e.g. \cite[Lemma 6.4.2]{CLS11}.
\end{proof}

\begin{remark}
    If the fan $\Sigma$ is regular, the multiplicity of all cones is $1$. Moreover, we can assume without loss of generality that $b_1 = 1$ and thus also $b_{n+1} = 1$ by the above proposition. Hence the formula in \Cref{prop:intersectionformula} reduces to \begin{equation}\label{eq:reducedintersectionformula}
        D_\rho \cdot V(\tau) = \begin{cases}
            0 & \rho \notin \{\rho_1,\dots,\rho_{n+1}\}, \\
            1 & \rho = \rho_1, \ \rho_{n+1}, \\
            b_i & \rho = \rho_i, \quad i \neq 1, n+1. \\
        \end{cases}
    \end{equation}
\end{remark}

\subsection{A particular toric singularity}\label{sec:toricresolution}

We discuss the toric description of a particular singularity. This type of singularity appears naturally in our problem: We consider a strictly semi-stable $R$-scheme $\mathfrak{X} \to \Spec R$ for some discrete valuation ring $R$. Then \'etale locally at a point of the special fibre the scheme $\mathfrak{X}$ looks like
\begin{equation}\label{eq:localX}
    \{t - x_1 \cdots x_n = 0\} \subset \aff^{n}_R,
\end{equation}
where $t \in m_R \subset R$ is a uniformizer and $n \geq 1$. We are interested in Hartl's resolution after some finite, ramified  base change  $\Tilde{R}/R$. The base-change $\mathfrak{X}_{\Tilde{R}} \to \Spec \Tilde{R}$ admits a resolution to a strictly semi-stable $\Tilde{R}$-scheme by multiple blow-ups of the irreducible components, see \Cref{prop:Hartlresolution}. During the blow-ups the local equation looks like \eqref{eq:localequations} below, see \cite[proof of Proposition 2.2]{Har01}. Thus the following lemma is a toroidal description of the behaviour under these blow-ups.

\begin{lemma}\label{lem:localmodel}
    For any $m, n \geq 1$ and $r_1,\dots,r_m \geq 1$, the affine variety \begin{equation}\label{eq:localequations}
        Z =\left\{x_1^{r_1} x_2^{r_2} \cdots x_m^{r_m} - y_0 y_1 \cdots y_n = 0\right\} \subset \aff^{n+m+1}_k
    \end{equation}
    is the affine toric variety corresponding to the fan spanned by the cone $$
        \sigma = \sigma_{n,r_1,\dots,r_m} = \Cone\left(\{e_i : i = 1,\dots,m\} \cup \{e_i + r_i f_j : i = 1,\dots,m, \ j = 1,\dots,n\}\right)
    $$
    in $N_\R = \R^{m+n}$, where $e_1,\dots,e_{m},f_1,\dots,f_n$ is a basis of $N = \Z^m \oplus \Z^n$. Moreover, 
    \begin{enumerate}[label=(\alph*)]
        \item The rays $\rho_{i,j} = \R_{\geq 0} (e_i + r_i f_j)$ and $\rho_{i,0} = \R_{\geq 0} (e_i)$ for $i = 1,\dots,m$ and $j = 1,\dots,n$ correspond to the irreducible subvarieties $D_{\rho_{i,j}} = V(x_i,y_j)$ for $i = 1,\dots,m$ and $j = 0,1,\dots,n$.
        \item For $i = 1,\dots,m$ the natural projection $\overline{\pi} = \pr_i \oplus 0 \colon \Z^m \oplus \Z^{n} \to \Z$ onto the $i$-th coordinate is compatible with the fan given by $\sigma$ and $\R_{\geq 0} \cdot 1 \subset N'_\R = \R$ and the corresponding toric morphism is the projection onto the $i$-th coordinate $x_i$: $\pi = \pr_i \colon Z \to \aff^1_k$.
        \item For $r_1 \geq 1$, the blow-up of $Z$ along $V(x_1,y_1)$ is given by the fan spanned by the two cones
        $$
            \begin{aligned}
                \sigma' &= \Cone(\{e_i : i = 1,\dots,m+1\} \cup \{e_i + r_i f_j : i = 1,\dots,m+1, \ j = 2,\dots,n\}) \\
                \sigma'' &= \Cone(\{e_i : i = 2,\dots,m+1\} \cup \{e_i + r_i f_j : i = 2,\dots,m+1, \ j = 1,\dots,n\}),
            \end{aligned}
        $$
        where $e_{m+1} := e_1 + f_1$ and $r_{m+1} := r_1 - 1$.
    \end{enumerate}
\end{lemma}

\begin{remark}
    For $m = 2$ and $r_1 = r_2 = 1$, \cite[Lemma 2.2]{Shi19} provides a similar computation.
\end{remark}

\begin{proof}
    Let $\sigma$ be the cone as defined in the statement. We claim that the dual cone $\sigma^\vee$ is given by $$
        \sigma^\vee = \Cone(e_1^\ast,\dots,e_{m}^\ast,f_1^\ast,\dots,f_n^\ast,r_1 e_1^\ast + \dots + r_m e_m^\ast - f_1^\ast - \dots - f_{n}^\ast). 
    $$
    If this claim is true, the first statement follows immediately, because
    $$
        k[\sigma^\vee \cap M] = \sfrac{k[x_1,\dots,x_m,y_0,\dots,y_n]}{(x_1^{r_1}\cdots x_m^{r_m} - y_0y_1\cdots y_n)}.
    $$
    To prove the claim, note that ``$\supseteq$'' is obvious. We prove the other inclusion: Let $$
        \sum\limits_{i=1}^{m} a_i e_i^\ast + \sum\limits_{j=1}^n b_j f_j^\ast \in \sigma^\vee.
    $$
    By definition, the coefficients $a_i,b_j$ satisfy
    $$
        a_i \geq 0, \quad a_i + r_i b_j \geq 0, \ \text{for } i = 1,\dots,m \text{ and } j = 1,\dots,n.
    $$
    Let $- \lambda = \min\{b_1,\dots,b_{n},0\} \leq 0$, then
    $$
        \sum\limits_{i=1}^{m} a_i e_i^\ast + \sum\limits_{j=1}^n b_j f_j^\ast = \underbrace{\lambda}_{\geq 0} (r_1 e_1^\ast + \dots + r_m e_m^\ast - f_1^\ast - \dots - f_{n}^\ast) + \sum\limits_{i = 1}^m \underbrace{(a_i - \lambda r_i)}_{a_i + r_i b_j \text{ or } a_i \geq 0} e_i^\ast + \sum\limits_{j = 1}^n \underbrace{(\lambda + b_j)}_{\geq 0} f_j^\ast.
    $$
    This shows ``$\subseteq$'' and thus the claim.
    
    We prove item (a): Note that the orthogonal part of the cones $\rho_{i,j}$ for $i = 1,\dots,m$ and $j = 0,\dots,n$ are given by
    $$
        \rho_{i,j}^\perp = \begin{cases}
            \Cone\left(e_1^\ast,\dots,\widehat{e_i^\ast},\dots,e_m^\ast,f_1^\ast,\dots,f_n^\ast\right) & \text{if } j = 0, \\
            \Cone\left(e_1^\ast,\dots,\widehat{e_i^\ast},\dots,e_m^\ast,f_1^\ast,\dots,\widehat{f_j^\ast},\dots,f_n^\ast,r_1 e_1^\ast + \dots r_m e_m^\ast - f_1^\ast - \dots - f_n^\ast \right) & \text{if } j \neq 0. \\
        \end{cases}
    $$
    Hence, the distinguished points of the cones $\rho_{i,j}$ are the points $(x_1,\dots,x_m,y_0,y_1,\dots,y_n) \in \aff^{n+m+1}_k$ with $x_a = \delta_{a,i}$ and $y_b = \delta_{b,j}$ where $\delta_{\cdot,\cdot}$ is the Kronecker delta. Thus statement (a) follows from the orbit-cone correspondence. Statement (b) follows directly from the construction.

    For statement (c), note that the blow-up of \eqref{eq:localequations} along $V(x_1,y_1)$ is given by 
    $$
        \begin{aligned}
            x_1 &= x_1' y_1, \quad (x_1')^{r_1}x_2^{r_2} \cdots x_m^{r_m}y_1^{r_1-1} - y_0y_2 \cdots y_n = 0, \\
            y_1 &= y_1' x_1, \quad (x_1)^{r_1 - 1} x_2^{r_2} \cdots x_m^{r_m} - y_0 y_1'y_2 \cdots y_n = 0.
        \end{aligned}
    $$
    Thus we find that the corresponding dual vectors are given by $$
    \begin{aligned}
        (e_1')^\ast &= e_1^\ast - f_1^\ast, \ (e_i')^\ast = e_i^\ast, \ (e_{m+1}')^\ast = f_1^\ast, \ (f_j')^\ast = f_j^\ast, \quad i = 2,\dots,m,\ j = 2,\dots,n, \\
        (f_1'')^\ast &= f_1^\ast - e_1^\ast, \ (e_i'')^\ast = e_i^\ast, \ (f_j'')^\ast = f_j^\ast, \quad i = 1,\dots,m,\ j = 2,\dots,n,
    \end{aligned}
    $$
    i.e.
    $$
    \begin{aligned}
        e_{m+1}' &= e_1 + f_1, \ e_i' = e_i, \ f_j' = f_j, \quad i = 1,\dots,m, \ j = 2,\dots,n, \\
        e_1'' &= e_1 + f_1, \ e_i'' = e_i, \ f_j'' = f_j, \quad i = 2,\dots,m, \ j =1,\dots,n.
    \end{aligned}
    $$
    Thus the description for $\sigma'$ and $\sigma''$ follows from the description of $Z$ in \eqref{eq:localequations}.
\end{proof}

\begin{remark}
    Up to relabeling the coordinates the blow-up description in (c) also holds for the blow-up along $V(x_i,y_j)$ where we think of $f_0$ as $f_0 = 0$ and do the calculation formally. More precisely, the blow-up of $V(x_1,y_0)$ corresponds to the subdivision of $\sigma$ into the two cones
    $$
        \begin{aligned}
            \sigma' &= \Cone\left(\{e_i + r_i f_j: i = 1,\dots,m+1, \ j = 1,\dots,n\}\right), \\
            \sigma'' &= \Cone\left(\{e_2,\dots,e_{m+1}\} \cup \{e_i + r_i f_j : i = 2,\dots,m+1, \ j = 1,\dots,n\}\right),
        \end{aligned}
    $$
    where $e_{m+1} = e_1$ and $r_{m+1} = r_1 - 1$.
\end{remark}

In concrete examples one can work with an explicit (log) resolution. For our purposes it suffices to work with some (log) resolution and use its properties. The first such property is the following.

\begin{proposition}\label{prop:refinement}
    Let $\sigma = \sigma_{n,r}$ be a cone from \Cref{lem:localmodel} with $m = 1$. Then any sequence of subdivisions of the same form as in \Cref{lem:localmodel} (c) terminates with a regular fan $\Tilde{\Sigma}$. Moreover, every minimal generator of a ray in $\Tilde{\Sigma}$ lies in the hyperplane $\{e_1 = 1\}$ and is a lattice point of the lattice generated by $e_1,f_1,f_2,\dots,f_n$ as defined in \Cref{lem:localmodel}.
\end{proposition}

\begin{remark}
    Note that the subdivision from \Cref{lem:localmodel} (c) corresponds to the blow-up of the irreducible components of the special fibre which is the procedure in \Cref{prop:Hartlresolution}.
\end{remark}

\begin{proof}
    The termination of this process is precisely \cite[Proposition 2.2]{Har01} which yields a resolution after some finite ramified base change, see also \Cref{prop:Hartlresolution}. Since the total space $\Tilde{\mathfrak{X}}$ of a strictly semi-stable model is regular, the local model of the toric variety is regular. Alternatively, the explicit local descriptions and the termination condition of the algorithm in \cite[Proposition 2.2]{Har01} imply also the regularity of the refined fan.

    The last statement follows from \Cref{lem:localmodel} (c) by induction. Indeed, starting with $\sigma = \sigma_{n,r}$ for some $n, r \in \N$ we see that the minimal generators of each rays is of the form $e_1$ or $e_1 + r f_j$ for some $j = 1,\dots,n$, i.e. they lie in the hyperplane $\{e_1 = 1\}$ and are lattice points. In the $j$-th step, we add a new vector $e_{j+1} := e_{j'} + f_{k'}$ for some $j' = 1, \dots,j$. $k' = 1,\dots,n$. The new minimal generators after the blow-up are of the form $$
        e_{j+1} + r_{j+1} f_{k}
    $$
    for some $k = 1,\dots,n$ and $r_{j+1} \in \Z_{\geq 0}$. Thus, we see by induction that all minimal generators lie in the hyperplane $\{e_1 = 1\}$ and are lattice points.
\end{proof}

\section{Chow group of the resolution}\label{sec:CH1resolution}

We recall the statement of \Cref{thm:main} and the strategy laid out in the introduction. Let $\mathfrak X \to \Spec R$ be a strictly semi-stable scheme over a dvr $R$ with algebraically closed residue field. Assume that the geometric generic fibre of $\mathfrak X \to \Spec R$ admits a decomposition of the diagonal. Then the complex
\begin{equation}\label{eq:complex2}
    \bigoplus\limits_{j \in I} \CH_1\left(Y_j\right) \overset{\Phi}{\longrightarrow} \CH_0\left(Y_i\right) \xrightarrow[]{\sum_i \deg} \Z \longrightarrow 0
\end{equation}
is exact (after any base change to a field extension $L'/k$), where $\Phi = \sum\limits_{i \in I} \sum\limits_{j \in I} \iota_i^\ast (\iota_j)_\ast$ with $\iota_i \colon Y_i \hookrightarrow \mathfrak X$ the natural inclusions. The assumption on the geometric generic fibre implies the exactness of the complex for a strictly semi-stable family $\Tilde{\mathfrak X} \to \Spec \Tilde{R}$ which is obtained by a finite (possibly ramified) base change and a resolution as in \Cref{prop:Hartlresolution}. We want to relate the complex \eqref{eq:complex2} for $\Tilde{\mathfrak X}$ with the complex for $\mathfrak X$ such that we can deduce from the exactness of the first the exactness of the latter, which then proves \Cref{thm:main}. In this section, we describe the special fibre $\Tilde{Y}$ of $\Tilde{\mathfrak X} \to \Spec \Tilde{R}$ and express its Chow group of one-cycles in terms of cycles supported $Y$. Before that we introduce some necessary notation regarding simplicial complexes and we define \emph{refinements} of them, which are a good frame work for a description of $\CH_1(\Tilde{Y})$. This description is then used in \Cref{sec:analysis} to relate the complexes of $\mathfrak X$ and $\Tilde{\mathfrak X}$.

\begin{definition}[Abstract simplicial complex, see {\cite[page 15]{Mun84}}]\label{def:simplicialcomplex}
\begin{enumerate}[label=(\roman*)]
    \item An \emph{abstract simplicial complex} $\mathcal{C}$ is a collection of non-empty finite sets, called \emph{simplices}, such that any non-empty subset $\sigma' \subset \sigma$ of a simplex $\sigma \in \mathcal{C}$ is contained in $\mathcal{C}$. We call a subset $\sigma' \subset \sigma$ a \emph{face} of $\sigma$ and a singleton set $\{v\} \in \mathcal{C}$ is called a \emph{vertex}.
    \item The dimension $\dim \sigma$ of a simplex $\sigma \in \mathcal{C}$ is defined as
    $$
        \dim \sigma = \# \sigma - 1,
    $$
    where $\# \sigma$ is the cardinality of the set $\sigma$. A $d$-dimensional simplex $\sigma$ is called \emph{d-simplex} and the set of $d$-simplices is denoted by $\mathcal{C}(d)$. For a vertex $\{v\} \in \mathcal{C}(0)$, we usually write $v \in \mathcal{C}(0)$.
    \item A subcollection $\mathcal{C}'$ of an abstract simplicial complex $\mathcal{C}$ is called an \emph{abstract simplicial subcomplex} (or simply \emph{subcomplex}) if $\mathcal{C}'$ is itself an abstract simplicial complex. In particular, for every $m \in \N$ the subcollection of all simplices of $\mathcal{C}$ of dimension at most $m$ is a subcomplex which we call the $m$-skeleton of $\mathcal{C}$.
\end{enumerate}
\end{definition}

Recall from graph theory: A vertex $v'$ in a graph is adjacent to a vertex $v$ if there exists an edge with endpoints $v$ and $v'$. We extend this definition to simplicial complexes by considering the $1$-skeleton of the simplicial complex which is a graph. 

\begin{definition}[Adjacent vertices]\label{def:adjacentvertices}
    Let $\mathcal{C}$ be a simplicial complex and let $v \in \mathcal{C}(0)$ be a vertex. We define the \emph{set of adjacent vertices} of $v$ in $\mathcal{C}$ as \begin{equation}\label{eq:adjacentvertices}
        A_{\mathcal{C}}(v) := \{w \in \mathcal{C}(0) : \{v,w\} \in \mathcal{C}(1)\}.
    \end{equation}
\end{definition}

Moreover, we define walls of a simplicial complex inspired by the definition of walls for toric varieties, see e.g. \cite[Section 10.5]{Dan78}.

\begin{definition}[Walls]
    Let $\mathcal{C}$ be a simplicial complex of dimension $n$, i.e. every simplex of $\mathcal{C}$ has dimension at most $n$ and at least one simplex has dimension $n$. A simplex $\tau \in \mathcal{C}(n-1)$ is a \emph{wall} of $\mathcal{C}$ if $\tau$ is a common face of two simplices $\sigma_1, \sigma_2 \in \mathcal{C}(n)$. Equivalently if $\tau$ is the (set-theoretic) intersection of the simplices $\sigma_1, \sigma_2$.
\end{definition}

The next definition is a slight adaption of the definition of quasi-geometric simplicial subdivisions in the sense of \cite[Definition 4.1 a) and first paragraph in section 2]{Sta92}. The assumptions on the geometric realization of the simplicial complexes are removed in our definition as our argument involves only the abstract simplicial complex. The name is borrowed from toric geometry as we construct a refinement of an simplicial complex by using local toric refinements, see \Cref{const:refinement}.

\begin{definition}[Refinement of an abstract simplicial complex]\label{def:refinement}
    Let $\mathcal{C}$ be an abstract simplicial complex. For a simplex $\sigma \in \mathcal{C}$, we denote by $\mathcal{C}_\sigma$ the subcomplex of $\mathcal{C}$ given by the simplex $\sigma$ and all its faces. A \emph{refinement} of $\mathcal{C}$ is an abstract simplicial complex $\mathcal{C}'$ together with a map of sets $\psi \colon \mathcal{C}' \to \mathcal{C}$ such that for all $\sigma \in \mathcal{C}$, the preimage $\psi^{-1}(\mathcal{C}_\sigma)$ is a non-empty finite simplicial subcomplex of $\mathcal{C}'$ of dimension at most $\dim \sigma$. To simplify the notation, we usually write $\psi^\ast(\sigma)$ instead of $\psi^{-1}(\mathcal{C}_\sigma)$.
\end{definition}

For such a refinement we define a relative version of adjacent vertices and walls which become useful later on.

\begin{definition}\label{def:notationSimplices}
    Let $\psi \colon \mathcal{C}' \to \mathcal{C}$ be a refinement of the abstract simplicial complex $\mathcal{C}$ in the above sense.
    \begin{enumerate}[label=(\roman*)]
        \item For any vertex $v' \in \mathcal{C}'(0)$ we define the \emph{set of relative adjacent vertices}
        $$
            A_\psi(v') := \begin{cases}
                A_{\mathcal{C}}(\psi(v')) &\text{if } \psi(v') \in \mathcal{C}(0), \\
                \emptyset & \text{otherwise,}
            \end{cases}
        $$
        where $A_{\mathcal{C}}$ is defined in \eqref{eq:adjacentvertices}.
        \item A \emph{relative wall} $\tau \in \mathcal{C}'$ of $\psi$ is a wall of the subcomplex $\psi^\ast(\sigma)$ for some simplex $\sigma \in \mathcal{C}$ such that $\dim \psi^\ast(\sigma) = \dim \sigma$. We denote the set of relative walls of the refinement $\psi \colon \mathcal{C}' \to \mathcal{C}$ by $\Wall(\psi)$.
    \end{enumerate}
\end{definition}

\begin{definition}[Associated simplicial complex]\label{def:associatedcomplex}
    Let $\mathfrak{X} \to \Spec R$ be a strictly semi-stable $R$-scheme and let $Y = \bigcup_{i \in I} Y_i$ be the special fibre of the family $\mathfrak{X} \to \Spec R$. We define the \emph{simplicial complex associated to} $\mathfrak{X} \to \Spec R$ as the simplicial complex $\mathcal{C}_{\mathfrak{X}}$ of the finite set $I$ satisfying $J \in \mathcal{C}_\mathfrak{X}$ for $J \subset I$ if and only if $\bigcap\limits_{j \in J} Y_j \neq \emptyset$. For any simplex $J \in \mathcal{C}_{\mathfrak{X}}$ we define $Y_{J}$ as the corresponding intersection $\bigcap\limits_{j \in J} Y_j$.
\end{definition}

\begin{remark}\label{rem:observation}
    A simplex $\sigma \in \mathcal{C}_{\mathfrak X}$ represents by definition an intersection $Y_J = \bigcap\limits_{j \in J} Y_j$.Then the intersection $Y_J$ is \'{e}tale locally the (affine) toric variety associated to the cone $\sigma_{n,1}$ from \Cref{lem:localmodel} with $m = 1$, where $n = \dim \sigma$.
\end{remark}

\begin{construction}\label{const:refinement}
    Let $\mathfrak{X} \to \Spec R$ be a strictly semi-stable $R$-scheme and $\Tilde{R}/R$ be a finite ramified extension of dvr's with ramification index $r$. Then the base-change $\mathfrak{X}_{\Tilde{R}} \to \Spec \Tilde{R}$ corresponds \'{e}tale locally to the cone $\sigma_{n,r}$ from \Cref{lem:localmodel} with the obvious morphism of lattices, namely scaling each generator of $N$ by $\frac{1}{r}$ except the first one. (Note that this morphism is compatible with the cones $\sigma_{n,r}$ and $\sigma_{n,1}$). Moreover, a resolution $\Tilde{\mathfrak{X}} \to \Spec \Tilde{R}$ of $\mathfrak{X}_{\Tilde{R}}$ in \Cref{prop:Hartlresolution} corresponds to a refinement as in \Cref{prop:refinement}. In particular, we find that the simplicial complex $\mathcal{C}_{\Tilde{\mathfrak{X}}}$ associated to the strictly semi-stable family $\Tilde{\mathfrak X} \to \Spec \Tilde{R}$ is a refinement of the (abstract) simplicial complex $\mathcal{C}_{\mathfrak{X}}$ as defined in \Cref{def:refinement} and we denote the map by $\psi \colon \mathcal{C}_{\Tilde{\mathfrak{X}}} \to \mathcal{C}_{\mathfrak{X}}$. Indeed, each cone associated to a simplex in $\mathcal{C}_{\mathfrak{X}}$ is replaced by a fan after the resolution which yields a simplicial complex by \Cref{prop:refinement}.
\end{construction}

\begin{remark}
    Note that this correspondence does not hold in each step of the refinement as the cones obtained along this process are not simplicial in general. Hence, we use the (global) description with simplicial complexes only for the strictly semi-stable model $\Tilde{\mathfrak{X}} \to \Spec \Tilde{R}$ and $\mathfrak{X} \to \Spec R$.
\end{remark}

\begin{remark}\label{rem:uniquevertex}
    From the above construction, we see that for every vertex $v \in \mathcal{C}_{\mathfrak X}$ there exists a unique vertex $\Tilde{v} \in \mathcal{C}_{\Tilde{\mathfrak X}}$ such that $\psi(\Tilde{v}) = v$. We denote this unique vertex usually by $\psi^\ast(v)$.
\end{remark}

We describe the new components appearing in the resolution process of \Cref{prop:Hartlresolution}, which we described \'{e}tale locally in \Cref{sec:toricresolution}. This enables us to explicitly describe the Chow group of one-cycles of the special fibres of the resolution, see \Cref{prop:CH1}.

\begin{proposition}\label{prop:simplicialcomplex}
    With the same notation as in \Cref{setup}, let $\psi \colon \mathcal{C}_{\Tilde{\mathfrak X}} \to \mathcal{C}_{\mathfrak X}$ be the refinement of simplicial complexes constructed in \Cref{const:refinement}. Recall that each component in $Y$ and $\Tilde{Y}$ corresponds to a vertex in $\mathcal{C}_{\mathfrak X}$ and $\mathcal{C}_{\Tilde{\mathfrak X}}$, respectively. For every $\Tilde{v} \in \mathcal{C}_{\Tilde{\mathfrak X}}(0)$, the irreducible component $\Tilde{Y}_{\Tilde{v}}$ is obtained by a finite sequence of blow-ups in smooth centers and Zariski locally trivial $\proj^1$-bundles of the base-change $(Y_{\psi(\Tilde{v})})_L$ where $$
        Y_{\psi(\Tilde{v})} = \bigcap\limits_{v \in \psi(\Tilde{v})(0)} Y_v.
    $$
\end{proposition}

\begin{proof}
    This follows from the description of the resolution process in \cite[Proposition 2.2]{Har01} together with the construction of the simplicial complex $\mathcal{C}_{\mathfrak X}$ and the toric construction of its refinement. Indeed, Hartl's resolution process is iterative and in each step every component will be blown-up in a smooth (possibly empty) center and the new components in that step are \'{e}tale locally trivial $\proj^1$-bundles over the intersection of two components in the previous step, see \Cref{lem:localmodel} or \cite[Proof of Proposition 2.2]{Har01}. Since we blow-up in that step one of the two components, we see that the \'{e}tale locally trivial $\proj^1$-bundle admits a section. Thus it is a Zariski-locally trivial $\proj^1$-bundle over the intersection of two components in the previous step. Hence each component $\Tilde{Y}_{\Tilde{v}}$ for $\Tilde{v} \in \mathcal{C}_{\Tilde{\mathfrak X}}(0)$ is of the form described in the statement.
\end{proof}

\begin{definition}\label{def:notationNewParts}
    With the same notation as in \Cref{prop:simplicialcomplex}. Let $\tau \in \Wall(\psi)$ be a relative wall as defined in \Cref{def:notationSimplices} (ii). Let $\psi(\tau)$ be the image of $\tau$ under the refinement morphism $\psi \colon \mathcal{C}_{\Tilde{\mathfrak{X}}} \to \mathcal{C}_{\mathfrak{X}}$, i.e. the smallest cone containing $\tau$. Then $\tau$ corresponds via the toric description to a (Zariski) locally trivial $\proj^1$-bundle $P_\tau$ over $(Y_{\psi(\tau)})_L$. We denote the bundle projection by $\Tilde{q}_\tau \colon P_\tau \to (Y_{\psi(\tau)})_L$.

    Additionally, we fix for any $\tau \in \Wall(\psi)$ a vertex $v(\tau) \in \mathcal{C}_{\Tilde{\mathfrak{X}}}(0)$ such that $v(\tau)$ is a face of $\tau$. This implies via the toric description that the projective bundle $P_\tau$ lies inside $\Tilde{Y}_{v(\tau)}$ and we denote the natural inclusion by $\Tilde{\iota}_\tau \colon \Tilde{Y}_{v(\tau)} \hookrightarrow \Tilde{Y}$.
\end{definition}

\begin{proposition}\label{prop:CH1}
    With the same notation as in \Cref{prop:simplicialcomplex} and \Cref{def:notationNewParts}.
    \begin{enumerate}[label=(\alph*)]
        \item The Chow group $\CH_1(Y)$ of one-cycles on $Y$ is generated by cycles of the form
        \begin{equation}\label{eq:onecycleY}
            \gamma = \sum\limits_{v \in \mathcal{C}_{\mathfrak{X}}(0)} (\iota_v)_\ast \gamma_v \in \CH_1(Y),
        \end{equation}
        where $\gamma_v \in \CH_1(Y_v)$ is a one-cycle and $\iota_v \colon Y_v \hookrightarrow Y$ is the natural inclusion.
        \item The Chow group $\CH_1(\Tilde{Y})$ of one-cycles on $\Tilde{Y}$ is generated by cycles of the form
        \begin{equation}\label{eq:onecycleTildeY}
            \Tilde{\gamma} = \sum\limits_{v \in \mathcal{C}_{\mathfrak{X}}(0)} (\Tilde{\iota}_v)_\ast \Tilde{\gamma}_v + \sum\limits_{\tau \in
                \Wall(\psi)} (\Tilde{\iota}_\tau)_\ast \Tilde{q}_\tau^\ast \Tilde{\alpha}_\tau \in \CH_1(\Tilde{Y}),
        \end{equation}
        where $\Tilde{\gamma}_v \in \CH_1((Y_v)_L)$ is a one-cycle, $\Tilde{\iota}_v \colon \Tilde{Y}_v \hookrightarrow \Tilde{Y}$ is the natural inclusion, $\Tilde{\alpha}_\tau \in \CH_0\left((Y_{\psi(\tau)})_L\right)$ is a zero-cycle, and $\Tilde{\iota}_\tau \colon \Tilde{Y}_{v(\tau)} \hookrightarrow \Tilde{Y}$ and $\Tilde{q}_\tau \colon P_{\tau} \to (Y_{\psi(\tau)})_L$ are the natural morphisms, see also \Cref{def:notationNewParts}.
    \end{enumerate} 
\end{proposition}

\begin{remark}
\begin{enumerate}[label=(\alph*)]
    \item This says that the Chow group of one-cycles on $\Tilde{Y}$ is governed by $\CH_1(Y)$ and by a toric part. Indeed, we can see \eqref{eq:onecycleTildeY} as
    $$
        \Tilde{\gamma} = \underbrace{\sum\limits_{v \in \mathcal{C}_{\mathfrak{X}}(0)} (\Tilde{\iota}_v)_\ast \Tilde{\gamma}_v}_{\text{from } \CH_1(Y)} + \underbrace{\sum\limits_{\tau \in \Wall(\psi)} (\Tilde{\iota}_\tau)_\ast \Tilde{q}_\tau^\ast \Tilde{\alpha}_\tau.}_{\text{``toric part''}}
    $$
    Morally speaking, the difference between \eqref{eq:onecycleY} and \eqref{eq:onecycleTildeY} lies in this toric part.
    \item Two examples of such a description can be found in \Cref{appendix}.
\end{enumerate}
\end{remark}

\begin{proof}
    Since any prime one-cycle, i.e. every irreducible, 1-dimensional subvariety, is contained in an irreducible component, statement (a) follows immediately. The same argument shows that $\CH_1(\Tilde{Y})$ is generated by one-cycles on the components of $\Tilde{Y}$. We show statement (b) by using the structure of the components from \Cref{prop:simplicialcomplex}: Recall the following two standard facts about Chow groups, see e.g. \cite[Theorem 3.3 (b) and Proposition 6.7 (e)]{Ful98}:
    \begin{enumerate}[label=(\roman*)]
        \item Let $P = \proj(E) \to W$ be the projectivization of a vector bundle $E$ on a smooth variety $W$ of relative dimension $\geq 2$. Then
        $$
            \CH_1(P) \cong \CH_0(W) \oplus \CH_1(W), \qquad \CH_0(P) \cong \CH_0(W),
        $$
        where the isomorphisms are given by pulling-back a cycle on $W$ along the flat map $P \to W$ and intersecting with a suitable power of the canonical line bundle $\mathcal{O}_{\proj(E)}(1)$ which is dual to the pull-back of the vector bundle $E$ on $W$ to $P = \proj(E)$.
        \item Let $\Tilde{W}$ be the blow-up of a smooth variety $W$ along a smooth subvariety $Z$. Then
        $$
            \CH_1(\Tilde{W}) \cong \CH_0(Z) \oplus \CH_1(W),
        $$
        where the isomorphism is given by applying the isomorphism from (i) to the projective bundle $\mathbb{P}(\mathcal{N}_{Z/W}) \to Z$ and pushing the one-cycle forward to $\Tilde{W}$ as well as pushing a one-cycle on $W \setminus Z$ forward along the natural morphism $W \setminus Z \to \Tilde{W}$. 
    \end{enumerate}
    Let $\Tilde{v} \in \mathcal{C}_{\Tilde{\mathfrak X}}(0)$ be a vertex. Assume first $\Tilde{v} = \psi^\ast(v) \in \psi^\ast \mathcal{C}_{\mathfrak X}(0)$. (Note that $\psi^\ast(v)$ consists of a single vertex by \Cref{rem:uniquevertex}.) By \Cref{prop:simplicialcomplex}, the component $\Tilde{Y}_{\Tilde{v}}$ of $\Tilde{Y}$ is a sequence of smooth blow-ups over $\left(Y_v\right)_L$, i.e. $$
        \Tilde{Y}_{\Tilde{v}} = \blowup_{Z_r} \blowup_{Z_{r-1}} \cdots \blowup_{Z_1} \left(Y_v\right)_L
    $$
    where $Z_1,\dots,Z_r$ are the smooth centers of the blow-ups. Then by fact (ii) we find 
    \begin{equation}\label{eq:isoCH1}
        \CH_1(\Tilde{Y}_{\Tilde{v}}) \cong \bigoplus\limits_{i=1}^r \CH_0(Z_i) \oplus \CH_1\left(\left(Y_v\right)_L\right).
    \end{equation}
    The $\CH_0(Z_i)$ parts correspond under the orbit-cone correspondence to some relative walls. Indeed, every zero-cycle is a (formal) linear combination of closed points. For each closed point the image under the isomorphism \eqref{eq:isoCH1} is the one-cycle given by pulling back the point to the corresponding $\proj^1$-bundle over $Z_i$ and pushing the one-cycle forward to $\Tilde{Y}_{\Tilde{v}}$ via the natural inclusion. This one-cycle is an irreducible rational curve which corresponds via the orbit-cone correspondence to a relative wall. Hence, we see that $\CH_1(\Tilde{Y}_{\Tilde{v}})$ is described by one-cycles corresponding to relative walls and $\CH_1\left(\left(Y_v\right)_L\right)$.

    We consider now the case $\Tilde{v} \notin \psi^\ast \mathcal{C}_{\mathfrak X}(0)$, i.e. the simplex $\sigma = \psi(\Tilde{v}) \in \mathcal{C}_{\mathfrak X}$ satisfies $\dim(\sigma) \geq 1$. By \Cref{prop:simplicialcomplex}, $\Tilde{Y}_{\Tilde{v}}$ is a sequence of smooth blow-ups and projective bundles over some $(Y_\sigma)_L$, i.e.
    $$
        \Tilde{Y}_{\Tilde{v}} = \blowup_{Z_r} \blowup_{Z_{r-1}} \cdots \blowup_{Z_1} P_{\Tilde{v}}
    $$
    where $Z_1,\dots,Z_r$ are the smooth centers of blow-ups and $P_{\Tilde{v}}$ is a Zariski locally trivial $\proj^1$-bundle over $W_{\Tilde{v}}$ such that the intersection with a different component $\Tilde{Y}_{\Tilde{v}'}$ yields a section. Then the facts (i) and (ii) above imply
    \begin{equation*}
        \CH_1(\Tilde{Y}_{\Tilde{v}}) \cong \bigoplus\limits_{i=1}^r \CH_0(Z_i) \oplus \CH_0(W_{\Tilde{v}}) \oplus \CH_1(W_{\Tilde{v}}).
    \end{equation*}
    The latter part $\CH_1(W_{\Tilde{v}})$ is contained in $\CH_1(Y_{\Tilde{v}'})$ and by the same argument as in the case $\Tilde{v} \in \psi^\ast \mathcal{C}_{\mathfrak X}(0)$ we find that the $\CH_0(Z_i)$ parts and $\CH_0(W)$ part correspond via the orbit-cone correspondence to some relative walls (in the above sense). Hence, we see that $\CH_1(\Tilde{Y}_{\Tilde{v}})$ is described by one-cycles corresponding to relative walls and the $\CH_1$ which appear in a previous step of \Cref{prop:Hartlresolution}. Inductively we get that the $\CH_1$ of the ``new'' components (corresponding to vertices in $\mathcal{C}_{\Tilde{\mathfrak X}}(0) \setminus \psi^{\ast}\mathcal{C}_{\mathfrak X}(0)$) is described by one-cycles corresponding to the relative walls and $\CH_1((Y_v)_L)$ for $v \in \mathcal{C}_{\mathfrak{X}}$. The one-cycles corresponding to the relative walls and the one-cycles supported on $(Y_v)_L$ might appear in multiple components, but they need to be counted once, as for an irreducible subvariety $Z$ of $\Tilde{Y}_v$ and $\Tilde{Y}_w$ the diagram with the natural inclusions
    $$
        \begin{tikzcd}
            Z \arrow[d,hook] \arrow[r,hook] & \Tilde{Y}_v \arrow[d,hook] \\
            \Tilde{Y}_w \arrow[r,hook] & \Tilde{Y},
        \end{tikzcd}
    $$
    commutes. This proves part (b) of the proposition.
\end{proof}

\section{Analysis of the base-change}\label{sec:analysis}

We recall our setup (\Cref{setup}), which we fix throughout this entire section. Let $R$ be a discrete valuation ring with residue field $k$ (not necessary algebraically closed) and let $\mathfrak X \to \Spec R$ be a strictly semi-stable $R$-scheme with special fibre $Y$ whose irreducible components are geometrically integral. Let $\Tilde{R}/R$ be a finite extension of dvr's with induced extension $L/k$ of residue fields. Let $\Tilde{\mathfrak X} \to \Spec \Tilde{R}$ be a strictly semi-stable $\Tilde{R}$-scheme with special fibre $\Tilde{Y}$ which is a resolution of the base change $\mathfrak{X}_{\Tilde{R}}$ from \Cref{prop:Hartlresolution} and let $q \colon \Tilde{\mathfrak X} \to \mathfrak X$ denote the natural morphism and also the restriction to the special fibre $q \colon \Tilde{Y} \to Y$. Moreover, let $\mathcal{C} := \mathcal{C}_{\mathfrak X}$ and $\Tilde{\mathcal{C}} := \mathcal{C}_{\Tilde{\mathfrak X}}$ be the simplicial complexes associated to the strictly semi-stable families $\mathfrak X \to \Spec R$ and $\Tilde{\mathfrak X} \to \Spec \Tilde{R}$, respectively (see \Cref{def:associatedcomplex}). Let $\psi \colon \Tilde{\mathcal{C}} \to \mathcal{C}$ be the refinement of simplicial complexes constructed in \Cref{const:refinement}. Consider the complex
\begin{equation}\label{eq:complexforanalysis}
    \bigoplus\limits_{j \in I} \CH_1(Y_j) \overset{\Phi}{\longrightarrow} \bigoplus\limits_{i \in I} \CH_0(Y_i) \xrightarrow[]{\sum_i \deg} \Z \longrightarrow 0,
\end{equation}
where $Y_i \subset Y$ are the irreducible components of $Y$ with $\iota_i \colon Y_i \hookrightarrow \mathfrak X$ the natural inclusions and $\Phi = \sum\limits_{i \in I} \sum\limits_{j \in I} \iota_i^\ast (\iota_j)_\ast$. In this section we relate the complex \eqref{eq:complexforanalysis} with the corresponding complex for $\Tilde{\mathfrak X}$. This relays on the following two auxiliary functions.

\begin{construction}\label{const:distance}
    We construct a function \begin{equation}\label{eq:distance}
        d \colon \Tilde{\mathcal{C}}(0) \times \mathcal{C}(0) \longrightarrow \Z_{\geq 0},
    \end{equation} which measures a ``distance'' between the vertices of the simplicial complex $\Tilde{C}$ to the vertices of $\mathcal{C}$.

    Let $v' \in \Tilde{\mathcal{C}}(0)$ be a vertex and let $\sigma = \psi(v') \in \mathcal{C}$. Then $\sigma$ corresponds by \Cref{const:refinement} to a cone $\sigma_{n,1}$ from \Cref{lem:localmodel} with $m = 1$. Let $w_0,w_1,\dots,w_n \in \sigma_{n,1}$ be the minimal generators of the cone $\sigma_{n,1}$, i.e. $w_i = e_1 + f_i$ in the notation of \Cref{lem:localmodel} (with $f_0 = 0$). The vertex $v'$ corresponds to a ray in the fan obtained by the resolution associated to $\psi$. Let $w'$ denote the minimal generator of this ray. Since this fan is a toric refinement of the cone $\sigma_{n,1}$ in the sense of \cite[before Example 3.3.12]{CLS11}, we can write
    $$
        w' = \sum\limits_{i=0}^n a_i w_i
    $$
    for some $a_i \geq 0$. As $\sigma_{n,1}$ (or equivalently $\sigma$) is simplicial, these $a_i$'s are unique. The $a_i$'s additionally satisfy \begin{equation}\label{eq:sumaiis1}
        \sum\limits_{i=0}^n a_i = 1,
    \end{equation}
    by \Cref{prop:refinement}. We define for any vertex $v \in \mathcal{C}(0)$
    \begin{equation}\label{eq:distancedef}
        d(v',v) = \begin{cases}
            r(1-a_i) & \text{if } v = v_i, \\
            r & \text{otherwise,}
        \end{cases}
    \end{equation}
    where $v_i \in \mathcal{C}(0)$ is the vertex associated to the minimal generator $w_i$ in $\sigma_{n,1}$. It remains to check that $d(v',v) \in \Z_{\geq 0}$ for all $v \in \mathcal{C}(0)$. If $v \neq v_i$ for all $i=0,1,\dots,n$, this is clear. For $i = 0,1,\dots,n$, we note that $0 \leq a_i \leq 1$, by \eqref{eq:sumaiis1}. Hence $d(v',v_i) \geq 0$ for all $i = 0,1,\dots,n$. Moreover, the point $w'$ is a lattice point in the sublattice generated by $e_1,\frac{1}{r}f_1,\frac{1}{r}f_2,\dots,\frac{1}{r}f_n$, see \Cref{prop:refinement}. Since $w_j = e_1 + f_j$, this immediately implies $a_i \in \frac{1}{r} \Z$. These two observations yield $d(v',v_i) \in \Z_{\geq 0}$ for all $i = 1,\dots,n$, i.e. we have a well-defined map of the form \eqref{eq:distance}.
\end{construction}

\begin{construction}\label{const:intersectionnumber}
    We construct a function \begin{equation}\label{eq:intersectionnumbers}
        I \colon \Wall(\psi) \times \Tilde{\mathcal{C}}(0) \longrightarrow \Z,
    \end{equation}
    which encodes the intersection numbers of one-cycles corresponding to the relative walls with the divisors corresponding to the vertices of the simplicial complex.

    Let $\sigma \in \mathcal{C}$ be a simplex and let $\tau' = \sigma'_1 \cap \sigma'_2 \in \Tilde{\mathcal{C}}$ be a wall of the simplicial subcomplex $\psi^\ast(\sigma)$, i.e. $\tau'$ is a relative wall of $\psi$. Let $d-1 = \dim \tau'$ be the dimension of the simplex $\tau'$. Let $v'_0,\dots,v'_{d}$ and $v'_1,\dots,v'_{d+1}$ be the vertices of $\sigma'_1$ and $\sigma'_2$, respectively. Recall once more that we associate to $\sigma$ the simplicial cone $\sigma_{d,1}$ from \Cref{lem:localmodel} with $m =1$. Then $\psi^\ast(\sigma)$ corresponds to a fan $\Sigma'$ which is a refinement of the cone $\sigma_{\dim \sigma,1}$. The simplices $\sigma_1', \sigma_2',$ and $\tau'$ correspond to cones of $\Sigma'$ of dimensions $d+1, d+1,$ and $d$ respectively. We denote the cones by $S_1', S_2',$ and $T'$ respectively. Note that $T' = S_1' \cap S_2'$. Recall also that the vertices $v'_0,\dots,v'_{d+1}$ correspond to rays $\rho'_0,\dots,\rho'_{d+1}$ of $S_1'$ or $S_2'$. We define for any $v' \in \Tilde{\mathcal{C}}(0)$
    \begin{equation}\label{eq:Intersectiondef}
        I_{\tau'}(v') := I(\tau',v') := \begin{cases}
            D_{\rho'_i} \cdot V(T') & \text{if } v' = v_i' \\
            0 & \text{otherwise,}
        \end{cases}
    \end{equation}
    where $D_{\rho'_i}$ is the divisor associated to the ray $\rho_i'$ and $V(T')$ is the Zariski closure of the orbit corresponding to $T'$ via the orbit-cone correspondence, see also \Cref{prop:intersectionformula}. It is obvious that $I_{\tau'}(v')$ is an integer, i.e. the function $I$ is well-defined as claimed in \eqref{eq:intersectionnumbers}.
\end{construction}

These two auxiliary functions satisfy the following crucial relation, which is deduced from the wall relation, see \eqref{eq:wallrelation}.

\begin{lemma}\label{lem:wallrelation}
    The functions $d$ defined in \Cref{const:distance} and $I$ defined in \Cref{const:intersectionnumber} satisfy for every $v \in \mathcal{C}(0)$ and $\tau' \in \Wall(\psi)$
    \begin{equation}\label{eq:wallrelation2}
            \sum\limits_{v' \in \Tilde{\mathcal{C}}(0)} \left(r-d(v',v)\right) I_{\tau'}(v') = 0.
    \end{equation}
\end{lemma}

\begin{proof}
    With the same notation as in \Cref{const:intersectionnumber}, let $w_0',\dots,w_{d+1}'$ be the minimal generators of the rays $\rho_0',\dots,\rho_{d+1}'$. The fan $\Sigma'$ is a toric refinement of the cone $\sigma_{d,1}$ associated to $\sigma$ and let $w_1,\dots,w_{d+1}$ be the minimal generators of the cone. We write for $i = 0,\dots,d+1$ the minimal generators $w'_i$ as a linear combination
    \begin{equation}\label{eq:basis}
        w'_i = \sum\limits_{j=1}^{d+1} a_{ij} w_j,
    \end{equation}
    which uniquely exists, since $\sigma_1'$ and $\sigma_2'$ are contained in the simplex $\sigma$. Moreover, by the wall relation \eqref{eq:wallrelation} there exist (up to scaling) unique $b_0,\dots,b_{d+1}$ such that $$
        \sum\limits_{i=0}^{d+1} b_i w_i' = 0.
    $$
    Together with \eqref{eq:basis} we find
    \begin{align*}
        0 &= \sum\limits_{i=0}^{d+1} b_i w_i'
        = \sum\limits_{i=0}^{d+1} b_i \left(\sum\limits_{j=1}^{d+1} a_{ij} w_j\right)
        = \sum\limits_{j=1}^{d+1} \left(\sum\limits_{i=0}^{d+1} a_{ij} b_i \right) w_j \\
        &= \lambda \sum\limits_{j=1}^{d+1} \left(\sum\limits_{i=0}^{d+1} (r-d(v_j,v_i')) I_{\tau'}(v_i') \right) w_j, 
    \end{align*} 
    where $\lambda = \frac{b_0}{r} \neq 0$ is the necessary scaling due to the choice with regard to the $b_j$'s. Note that we used in the last step the definition of $d$ and $I$ together with \eqref{eq:reducedintersectionformula}. Since $I_{\tau'}(v') = 0$ for every $v'$ different to all ($v_i'$)'s, this shows \eqref{eq:wallrelation2} because the cone $\sigma$ is simplicial and $\lambda \neq 0$.
\end{proof}

In order to relate the complex \eqref{eq:complexforanalysis} with the complex for the resolution $\Tilde{X}$ after a finite base change, we replace the term $\bigoplus_j \CH_1(Y_j)$ in the complex \eqref{eq:complexforanalysis} with $\CH_1(Y)$ which does not change the exactness of the complex, see \Cref{rem:nochangeincomplex}.

\begin{definition}[{\cite[Definition 3.1]{PS23}}]\label{def:PSmap}
    Let $\mathcal{X} \to \Spec R$ be a strictly semi-stable $R$-scheme with special fibre $Y$. Denote the irreducible components of $Y$ by $Y_i$ with $i \in I$. Then we define $$
        \Phi_{\mathcal{X},Y_i} \colon \CH_1(Y) \xrightarrow[]{\iota_\ast} \CH_1(\mathcal{X}) \xrightarrow[]{\iota_i^\ast} \CH_0(Y_i),
    $$
    where $\iota \colon Y \to \mathcal{X}$ and $\iota_i \colon Y_i \to \mathcal{X}$ are the natural inclusions. Moreover, we define
    $$
        \Phi_{\mathcal{X}} := \sum\limits_{i \in I} \Phi_{\mathcal{X},Y_i} \colon \CH_1(Y) \longrightarrow \bigoplus\limits_{i \in I} \CH_0(Y_i).
    $$
\end{definition}

\begin{remark}\label{rem:nochangeincomplex}
    The map $\Phi$ in the complex \eqref{eq:complexforanalysis} is the composition of the natural surjection $$
        \bigoplus\limits_{i \in I} \CH_1(Y_i) \longrightarrow \CH_1(Y),
    $$
    given by the pushforward along the inclusions $Y_i \hookrightarrow Y$, with the map $\Phi_{\mathfrak{X}}$, see also \cite[Lemma 3.2]{PS23}. Moreover, the images of $\Phi$ and $\Phi_{\mathfrak X}$ are equal, i.e. it suffices to consider $\Phi_{\mathfrak{X}}$.
\end{remark}

\begin{remark}\label{rem:imageofPSmap}
    With the same notation as in \Cref{def:PSmap}. Let $\gamma_i \in \CH_1(Y_i)$ be a one-cycle for some $i \in I$ and let $\iota_i' \colon Y_i \hookrightarrow Y$ be the natural inclusion, then for every $j \in I$ with $Y_i \cap Y_j \neq \emptyset$
    $$
        \Phi_{\mathcal{X},Y_l}\left(\left(\iota'_i\right)_\ast \gamma_i\right) = \begin{cases}
            - \left(\iota'_{i,j}\right)_\ast \restr{\gamma_i}_{Y_i \cap Y_j} \in \CH_0(Y_i), \quad &\text{if } l = i, \\
            \left(\iota'_{j,i}\right)_\ast \restr{\gamma_i}_{Y_i \cap Y_j} \in \CH_0(Y_j), \quad &\text{if } l = j, \\
            0 \in \CH_0(Y_l),\quad &\text{otherwise},
        \end{cases} 
    $$
    where $\iota'_{i,j} \colon Y_i \cap Y_j \hookrightarrow Y_i$ is the natural inclusion, see \cite[Lemma 3.2]{PS23}.
\end{remark}

Recall that we fixed throughout this section a strictly semi-stable family $\mathfrak X \to \Spec R$ over a dvr $R$ and a resolution $\Tilde{\mathfrak X} \to \Spec \Tilde{R}$ after a finite extension of dvr's $\Tilde{R}/R$ with induced extension $L/k$ of residue fields. We denote the special fibres by $Y$ and $\Tilde{Y}$, respectively. The associated simplicial complexes from \Cref{def:associatedcomplex} are denoted by $\mathcal{C}$ and $\Tilde{\mathcal{C}}$, respectively, and the refinement from \Cref{const:refinement} is denoted by $\psi \colon \Tilde{\mathcal{C}} \to \mathcal{C}$. Recall, for a simplex $\sigma \in \mathcal{C}$, we set
$$
    Y_\sigma := \bigcap\limits_{s \in \sigma(0)} Y_s.
$$

\begin{lemma}\label{lem:phiexplicit}
    The homomorphism $\Phi_{\mathcal X}$ from \Cref{def:PSmap} is given for the strictly semi-stable families $\mathfrak{X} \to \Spec R$ and $\Tilde{\mathfrak{X}} \to \Spec \Tilde{R}$ as follows:
    \begin{enumerate}[label=(\alph*)]
        \item For any $\gamma \in \CH_1(Y)$ of the form \eqref{eq:onecycleY} and for any $v \in \mathcal{C}(0)$,
        $$
            \Phi_{\mathfrak{X},Y_v}(\gamma) = \sum\limits_{w \in A_\mathcal{C}(v)} (\iota_{\{v,w\},v})_\ast \left(\restr{\gamma_w}_{Y_v \cap Y_w} - \restr{\gamma_v}_{Y_v \cap Y_w} \right) \in \CH_0(Y_v),
        $$
        where $\iota_{\{v,w\},v} \colon Y_v \cap Y_w \hookrightarrow Y_v$ is the natural inclusion and $A_{\mathcal{C}}(v)$ is defined in \eqref{eq:adjacentvertices}.
        \item For any $\Tilde{\gamma} \in \CH_1(\Tilde{Y})$ of the form \eqref{eq:onecycleTildeY} and for any $\Tilde{v} \in \Tilde{\mathcal{C}}(0)$,
        \begin{equation}\label{eq:Phiexplicit}
            \begin{aligned}
                \Phi_{\Tilde{\mathfrak{X}},\Tilde{Y}_{\Tilde{v}}}(\Tilde{\gamma}) &= \sum\limits_{v \in A_\psi(\Tilde{v})} - (\iota_{v;\Tilde{v}})_\ast \restr{\Tilde{\gamma}_{\psi(\Tilde{v})}}_{(Y_{v} \cap Y_{\psi(\Tilde{v})})_L} + \sum\limits_{\substack{(v,w) \in \mathcal{C}(0)^2: \\ \Tilde{v} \in \psi^\ast(\{v,w\}) \cap A_{\Tilde{\mathcal{C}}}(\psi^\ast(v))}} (\iota_{v,w;\Tilde{v}})_\ast \restr{\Tilde{\gamma}_v}_{(Y_v \cap Y_w)_L} \\ &\hphantom{=} + \sum\limits_{\tau \in \Wall(\psi)} (\Tilde{\iota}_{\tau,\Tilde{v}})_\ast I_\tau(\Tilde{v}) \Tilde{\alpha}_\tau
            \end{aligned}
        \end{equation}
        in $\CH_0(\Tilde{Y}_{\Tilde{v}})$, where $A_{\Tilde{\mathcal{C}}}(v)$ is defined in \Cref{def:adjacentvertices}, $\psi^\ast$ in \Cref{def:refinement}. $A_\psi(\Tilde{v})$ and $\Wall(\psi)$ in \Cref{def:notationSimplices}, and $I_{\tau}(\Tilde{v})$ in \eqref{eq:Intersectiondef}. Note that $\psi^\ast(v)$ is a unique vertex by \Cref{rem:uniquevertex}. Moreover, $\iota_{v;\Tilde{v}}$, $\iota_{v,w;\Tilde{w}}$, and $\Tilde{\iota}_{\tau,\Tilde{v}}$ are the natural inclusions, see also the remark below.
    \end{enumerate}
\end{lemma}

\begin{remark}
    \begin{enumerate}[label=(\alph*)]
        \item The zero-cycles appearing on the right hand side of \eqref{eq:Phiexplicit} lie in $(Y_{v} \cap Y_{\psi(\Tilde{v})})_L$ for $v \in A_\psi(\Tilde{v})$, $(Y_{v} \cap Y_{w})_L$ if $\Tilde{v} \in \psi^\ast(\{v,w\}) \cap A_{\mathcal{C}_{\Tilde{\mathfrak X}}}(v)$, or $(Y_{\psi(\tau)})_L$ for $\tau \in \Wall(\psi)$. If $I_\tau(\Tilde{v}) \neq 0$, then we know by the construction in \eqref{eq:Intersectiondef} that $\Tilde{v} \in \tau$, i.e. $\Tilde{v}$ is a vertex of the simplex $\tau$. Otherwise we disregard the term, i.e. those $\alpha_\tau$ with $I_\tau(\Tilde{v}) = 0$. 
        
        We know that a finite sequence of blow-ups in smooth centers and (Zariski) locally trivial $\proj^1$-bundles of $(Y_{v} \cap Y_{\psi(\Tilde{v})})_L$, $(Y_{v} \cap Y_{w})_L$, or $(Y_{\psi(\tau)})_L$ (for $I_\tau(\Tilde{v}) \neq 0$) lies in $\Tilde{Y}_{\Tilde{v}}$, see \Cref{prop:simplicialcomplex}. These are the natural inclusions $\iota_{v;\Tilde{v}}$, $\iota_{v,w;\Tilde{w}}$, and $\Tilde{\iota}_{\tau,\Tilde{v}}$, respectively. Note that this makes sense as the Chow group of zero-cycles is invariant under blow-ups in smooth centers and (Zariski) locally trivial $\proj^1$-bundles by \cite[Theorem 3.3 (b) and Proposition 6.7 (e)]{Ful98}.
        \item Two explicit examples of the map $\Phi_{\Tilde{\mathfrak X}}$ are provided in \Cref{appendix}.
    \end{enumerate}
\end{remark}

\begin{proof}
    Item (a) follows immediately from the explicit description of $\Phi_{\mathfrak{X}}$ in \Cref{rem:imageofPSmap} together with \eqref{eq:onecycleY}. For item (b), we can use the same argument, i.e. apply \Cref{rem:imageofPSmap} to the explicit description of the one-cycles in \eqref{eq:onecycleTildeY}. The only non-trivial part is the intersection
    \begin{equation}\label{eq:intersectionwithtau}
        \Tilde{Y}_{\Tilde{v}} \cdot \left(\Tilde{\iota}_\tau\right)_\ast\Tilde{q}^\ast_\tau \Tilde{\alpha}_\tau = \left(\Tilde{\iota}_{\tau;\Tilde{v}}\right)_\ast I_\tau(\Tilde{v}) \alpha_{\tau}.
    \end{equation}
    Since $\alpha_\tau$ is a linear combination of rational equivalent classes of points, it suffices to show \eqref{eq:intersectionwithtau} for every point in $\left(Y_{\psi(\tau)}\right)_L$. For each such point we can consider the \'{e}tale local toric description of $\Tilde{\mathfrak X}$ and find that $I_{\tau}(\Tilde{v})$ is by \eqref{eq:Intersectiondef} the intersection multiplicity at that point. Since this is independent of the point chosen, we obtain \eqref{eq:intersectionwithtau} and thus item (b).
\end{proof}

\begin{proposition}[Key formula]\label{prop:key}
    Let $\mathfrak{X} \to \Spec R$ be a strictly semi-stable $R$-scheme with special fibre $Y$, let $\Tilde{R}/R$ be a finite extension of dvr's of ramification index $r$ with induced extension $L/k$ of residue fields, and let $\Tilde{\mathfrak{X}} \to \Spec \Tilde{R}$ be a resolution of the base-change $\mathfrak{X}_{\Tilde{R}}$ from \Cref{prop:Hartlresolution} and let $\Tilde{Y}$ be its special fibre. Let $q \colon \Tilde{Y} \to Y$ be the natural morphism and let $\psi \colon \mathcal{C}_{\Tilde{\mathfrak X}} \to \mathcal{C}_{\mathfrak X}$ be the refinement of the simplicial complex associated to $\Tilde{\mathfrak X}$ and $\mathfrak X$ as constructed in \Cref{const:refinement}. Then for any $\Tilde{\gamma} \in \CH_1(\Tilde{Y})$ and $v \in \mathcal{C}_{\mathfrak{X}}(0)$ the following holds:
    \begin{equation}\label{eq:key}
        \Phi_{\mathfrak{X},Y_v}(q_\ast \Tilde{\gamma}) = \sum\limits_{\Tilde{v} \in \mathcal{C}_{\Tilde{\mathfrak{X}}}(0)} \left(\iota_{\psi(\Tilde{v}),v}\right)_\ast (r - d(\Tilde{v},v)) q_\ast \Phi_{\Tilde{\mathfrak{X}},\Tilde{Y}_{\Tilde{v}}}(\Tilde{\gamma}) \in \CH_0(Y_v),
    \end{equation}
    where $d(\Tilde{v},v)$ is defined in \eqref{eq:distancedef} and $\iota_{\psi(\Tilde{v}),v} \colon Y_{\psi(\Tilde{v})} \hookrightarrow Y_v$ is the natural inclusion for $d(\Tilde{v},v) < r$.
\end{proposition}

\begin{proof}
    Let $v \in \mathcal{C}_{\mathfrak{X}}(0)$ be a vertex corresponding to an irreducible component $Y_v$ of the special fibre $Y$ and let $\Tilde{\gamma} \in \CH_1(\Tilde{Y})$ be a one-cycle on $\Tilde{Y}$ which we can assume to be of the form \eqref{eq:onecycleTildeY}. Note first that $d(\Tilde{v},v) < r$ if and only if $v$ is a vertex of $\psi(\Tilde{v})$ by the construction in \eqref{eq:distancedef}. Hence the right hand side of \eqref{eq:key} is well-defined in $\CH_0(Y_v)$. We claim that the equality in \eqref{eq:key} is an equality of cycles. If $r = 1$, then there is nothing to prove, as the base-change and the resolution are trivial, i.e. we can assume without loss of generality $r > 1$. We note that by the projection formula \cite[Proposition 2.3 (c)]{Ful98} for every $w \in A_{\mathcal{C}_{\mathfrak X}}(v)$
    $$
        q_\ast \left(\restr{\Tilde{\gamma}_{v}}_{\left(Y_v \cap Y_{w} \right)_L}\right) = \left(\iota_{\{w,v\},v}\right)_\ast \left(\restr{\left(q_\ast \Tilde{\gamma}_{v} \right)}_{Y_v \cap Y_{w}}\right),
    $$
    where $\iota_{\{w,v\},v} \colon Y_v \cap Y_w \hookrightarrow Y_v$ is the natural inclusion. Similarly, for every $\Tilde{v} \in \psi^\ast(\{v,w\}) \cap A_{\mathcal{C}_{\Tilde{\mathfrak X}}}(v)$ for some $w \in A_{\mathcal{C}_{\mathfrak X}}(v)$ the projection formula yields
    $$
        \left(\iota_{\psi(\Tilde{v}),v}\right)_\ast q_\ast \left(\restr{\Tilde{\gamma}_{v}}_{\left(Y_v \cap Y_{w} \right)_L}\right) = \left(\iota_{\{w,v\},v}\right)_\ast \left( \restr{\left(q_\ast \Tilde{\gamma}_{v} \right)}_{Y_v \cap Y_{w}}\right),
    $$
    where again $\iota_{\{w,v\},v} \colon Y_v \cap Y_w \hookrightarrow Y_v$ is the natural inclusion. Note that the analogous formula holds true after switching $v$ and $w$. We count the appearance of each term on the right hand side of \eqref{eq:key}:

    For any $w \in A_{\mathcal{C}_{\mathfrak X}}(v)$, the term $\left(\iota_{\{w,v\},v}\right)_\ast \left( \restr{\left(q_\ast \Tilde{\gamma}_{v} \right)}_{Y_v \cap Y_{w}}\right)$ appears only in $\Phi_{\Tilde{\mathfrak{X}},\Tilde{Y}_v}(\Tilde{\gamma})$ and $\Phi_{\Tilde{\mathfrak{X}},\Tilde{Y}_{\Tilde{v}}}(\Tilde{\gamma})$ for the unique $\Tilde{v}$ in $\psi^\ast(\{v,w\}) \cap A_{\mathcal{C}_{\Tilde{\mathfrak X}}}(v)$. Hence we get on the right hand side
    $$
        r \left(- \left(\iota_{\{w,v\},v}\right)_\ast \restr{\left(q_\ast \Tilde{\gamma}_{v} \right)}_{Y_v \cap Y_{w}}\right) + (r-1) \left(\iota_{\{w,v\},v}\right)_\ast \restr{\left(q_\ast \Tilde{\gamma}_{v} \right)}_{Y_v \cap Y_{w}} = - \left(\iota_{\{w,v\},v}\right)_\ast \restr{\left(q_\ast \Tilde{\gamma}_{v} \right)}_{Y_v \cap Y_{w}}.
    $$
    Similarly, for any $w \in A_{\mathcal{C}_{\mathfrak X}}(v)$, the term $\left(\iota_{\{w,v\},v}\right)_\ast \left( \restr{\left(q_\ast \Tilde{\gamma}_{w} \right)}_{Y_v \cap Y_{w}}\right)$ appears only in $\Phi_{\Tilde{\mathfrak{X}},\Tilde{Y}_{\Tilde{v}}}(\Tilde{\gamma})$ for the unique $\Tilde{v}$ in $\psi^\ast(\{v,w\}) \cap A_{\mathcal{C}_{\Tilde{\mathfrak X}}}(w)$. Note that the term $\restr{\left(q_\ast \Tilde{\gamma}_{w} \right)}_{Y_v \cap Y_{w}}$ from $\Phi_{\Tilde{\mathfrak{X}},\Tilde{Y}_w}(\Tilde{\gamma})$ does not appear as its coefficient is $r-d(w,v) = r - r = 0$ by \eqref{eq:distancedef}. Thus the term $\left(\iota_{\{w,v\},v}\right)_\ast \left( \restr{\left(q_\ast \Tilde{\gamma}_{w} \right)}_{Y_v \cap Y_{w}}\right)$ appears with coefficient $1$ on the right hand side of \eqref{eq:key}. Lastly, the terms $\left(\iota_{\Tilde{v},v}\right)_\ast q_\ast \Tilde{\alpha}_{\tau}$ vanish on the right hand side by \Cref{lem:wallrelation}. 
    
    Since $\Phi_{\mathfrak{X},Y_v}(q_\ast \Tilde{\gamma})$ is given by
    $$
        \Phi_{\mathfrak{X},Y_v}(q_\ast \Tilde{\gamma}) = \sum\limits_{w \in A_{\mathcal{C}_{\mathfrak X}}(v)} (\iota_{\{v,w\},v})_\ast \left(\restr{\left(q_\ast \Tilde{\gamma}_w\right)}_{Y_v \cap Y_w} - \restr{\left(q_\ast \Tilde{\gamma}_v\right)}_{Y_v \cap Y_w} \right),
    $$
    the above argument shows the equality in \eqref{eq:key}.
\end{proof}

\section{Proof of the main results}\label{sec:mainresult}

\subsection{Exactness of the complex}

\begin{theorem}\label{thm:mainresult1}
    Let $R$ be a dvr with algebraically closed residue field $k$. Let $\mathfrak{X} \to \Spec R$ be a strictly semi-stable $R$-scheme with special fibre $Y$, let $\Tilde{R}/R$ be a finite extension of dvr's with ramification index $r$, and let $\Tilde{\mathfrak{X}} \to \Spec \Tilde{R}$ be a resolution of $\mathfrak{X}_{\Tilde{R}}$ from \Cref{prop:Hartlresolution}. Let $\Lambda = \Z/m $ for some $m \in \N_0$ and let $\psi \colon \mathcal{C}_{\Tilde{\mathfrak{X}}} \to \mathcal{C}_{\mathfrak{X}}$ be the refinement of simplicial complexes associated to the strictly semi-stable families $\mathfrak{X} \to \Spec R$ and $\Tilde{\mathfrak{X}} \to \Spec \Tilde{R}$ from \Cref{const:refinement}. If the complex
    \begin{equation}\label{eq:complexbasechange}
        \CH_1(\Tilde{Y}) \otimes_\Z \Lambda \xrightarrow[]{\Phi_{\Tilde{\mathfrak{X}}} \otimes \Lambda} \bigoplus\limits_{\Tilde{v} \in \mathcal{C}_{\Tilde{\mathfrak X}}(0)} \CH_0(\Tilde{Y}_{\Tilde{v}}) \otimes_\Z \Lambda \xrightarrow{\sum_{\Tilde{v}}\deg} \Lambda
    \end{equation}
    is exact, then the complex
    \begin{equation}\label{eq:complexbefore}
        \CH_1(Y) \otimes_\Z \Lambda \xrightarrow[]{\Phi_{\mathfrak{X}} \otimes \Lambda} \bigoplus\limits_{v \in \mathcal{C}_{\mathfrak X}(0)} \CH_0(Y_v) \otimes_\Z \Lambda \xrightarrow{\sum_v \deg} \Lambda
    \end{equation}
    is exact, where $\Phi_{\Tilde{\mathfrak{X}}}$ and $\Phi_{\mathfrak{X}}$ are defined in \Cref{def:PSmap}.
\end{theorem}

\begin{proof}
    Let $\beta = (\beta_v)_v \in \bigoplus\limits_{v \in \mathcal{C}_{\mathfrak{X}}} \CH_0(Y_v) \otimes \Lambda$ be of total degree $\sum\limits_{v} \deg (\beta_v) = 0$. Since \eqref{eq:complexbasechange} is exact, there exists a one-cycle $\Tilde{\gamma} \in \CH_1(\Tilde{Y}) \otimes \Lambda$ of the form \eqref{eq:onecycleTildeY} such that
    \begin{equation}\label{eq:firstsurjectivity}
        \left(\Phi_{\Tilde{\mathfrak{X}},\Tilde{Y}_{\Tilde{v}}} \otimes \Lambda\right) (\Tilde{\gamma}) = \begin{cases}
            \beta_{\psi(\Tilde{v})} & \text{if } \psi(\Tilde{v}) \in \mathcal{C}_{\mathfrak{X}}(0), \\
            0 & \text{otherwise.}
        \end{cases}
    \end{equation}
    Let $q \colon \Tilde{\mathfrak{X}} \to \mathfrak{X}$ be the natural morphism and let $q \colon \Tilde{Y} \to Y$ also denote the restriction to the special fibre. Then, we consider the element
    \begin{equation}\label{eq:beta'}
        \begin{aligned}
            \beta' := \beta - \left(\Phi_{\mathfrak{X}} \otimes \Lambda\right)(q_\ast \Tilde{\gamma}) &= \left(\beta_v - \left(\Phi_{\mathfrak{X},Y_v} \otimes \Lambda\right) (q_\ast \Tilde{\gamma})\right)_v \\
            &= \left(q_\ast \left(\Phi_{\Tilde{\mathfrak{X}},\Tilde{Y}_{\psi^\ast(v)}} \otimes \Lambda\right) (\Tilde{\gamma}) - \left(\Phi_{\mathfrak{X},Y_v} \otimes \Lambda\right) (q_\ast \Tilde{\gamma})\right)_v \\
        \end{aligned}
    \end{equation}
    in $\bigoplus\limits_v \CH_0(Y_v) \otimes \Lambda$, where $\psi^\ast(v) \in \mathcal{C}_{\Tilde{\mathfrak X}}(0)$ is the unique vertex mapping via $\psi$ to $v$, see \Cref{rem:uniquevertex}. Note that for $\Tilde{v} \in \mathcal{C}_{\Tilde{\mathfrak{X}}}(0) \setminus \psi^\ast\left(\mathcal{C}_{\mathfrak{X}}(0)\right)$, we have
    $$
        q_\ast \left(\Phi_{\Tilde{\mathfrak{X}},\Tilde{Y}_{\Tilde{v}}} \otimes \Lambda \right)(\Tilde{\gamma}) = q_\ast 0 = 0
    $$
    by \eqref{eq:firstsurjectivity}. By adding these terms to some choosen component (depending on $\Tilde{v}$), we find by \Cref{rem:imageofPSmap} that $\beta'$ can be written as a sum of elements $$
        \beta'' := (\beta''_v)_v \in \bigoplus\limits_{v \in \mathcal{C}_{\mathfrak{X}}(0)} \CH_0(Y_v) \otimes \Lambda
    $$
    of the form
    $$
        \beta''_v = \begin{cases}
            (\iota_{\{v_1,v_2\},v_1})_\ast \alpha & \text{if } v = v_1, \\
            - (\iota_{\{v_1,v_2\},v_2})_\ast \alpha & \text{if } v = v_2, \\
            0 & \text{otherwise,}
        \end{cases}
    $$
    for some $v_1,v_2 \in \mathcal{C}_{\mathfrak{X}}(0)$, $\alpha \in \CH_0(Y_{v_1} \cap Y_{v_2}) \otimes \Lambda$ and $\iota_{\{v,w\},w} \colon Y_{v} \cap Y_w \hookrightarrow Y_w$ the natural inclusion. We claim that $\beta'' \in \Ima \Phi_{\mathfrak X} \otimes \Lambda$. Indeed, consider $\Tilde{\gamma}'' \in \CH_1(\Tilde{Y}) \otimes \Lambda$ with
    $$
        \left(\Phi_{\Tilde{\mathfrak{X}},\Tilde{Y}_{\Tilde{v}}} \otimes \Lambda\right)(\Tilde{\gamma}'') = \begin{cases}
            \left(\Tilde{\iota}_{\{v_1,v_2\},v_1}\right)_\ast \alpha & \text{if } \Tilde{v} = v_1 \\
            - \alpha & \text{if } \Tilde{v} \in \psi^\ast(\{v_1,v_2\}) \cap A_{\mathcal{C}_{\Tilde{\mathfrak X}}}(v_1) \\
            0 & \text{otherwise,}
        \end{cases}
    $$
    which exists by the exactness of \eqref{eq:complexbasechange}. Then \Cref{prop:key} yields
    $$
        \begin{aligned}
            \left(\Phi_{\mathfrak X,Y_{v_1}}\otimes\Lambda\right)(q_\ast \Tilde{\gamma}'') &= r q_\ast \left(\Tilde{\iota}_{\{v_1,v_2\},v_1}\right)_\ast \alpha + (\iota_{\{v_1,v_2\},v_1})_\ast (r-1) q_\ast (-\alpha) + 0  = \left(\iota_{\{v_1,v_2\},v_1}\right)_\ast \alpha,\\
            \left(\Phi_{\mathfrak X,Y_{v_2}}\otimes\Lambda\right)(q_\ast \Tilde{\gamma}'') &= 0 + (\iota_{\{v_1,v_2\},v_1})_\ast \left(1 \cdot q_\ast (-\alpha)\right) = - (\iota_{\{v_1,v_2\},v_2})_\ast \alpha, \\
            \left(\Phi_{\mathfrak X,Y_{v}}\otimes\Lambda\right)(q_\ast \Tilde{\gamma}'') &= 0,
        \end{aligned}
    $$
    for $v \in \mathcal{C}_{\mathfrak X} \setminus \{v_1,v_2\}$. This shows the claim and thus $\beta \in \Ima \left(\Phi_{\mathfrak X} \otimes \Lambda \right)$.
\end{proof}

\begin{remark}
    Note that the assumption on the residue field being algebraically closed is crucial for this argument. Because otherwise, we need to consider $\left(\beta_{\psi(\Tilde{v})}\right)_L$ in \eqref{eq:firstsurjectivity}. But then $q_\ast \left(\beta_{\psi(\Tilde{v})}\right)_L = [L:k] \cdot \beta_{\psi(\Tilde{v})}$, i.e. we can only obtain a multiple of $\beta$ with this argument.
\end{remark}

\begin{remark}\label{rem:slightgeneralization}
    With the same notation as in \Cref{setup}, assume additionally that $k$ is algebraically closed, in particular $L = k$. Let $A/R$ be an unramified extension of dvr's with induced extension $L'/k$ of residue fields and let $\Tilde{A}/\Tilde{R}$ be an unramified extension of dvr's with induced extension $L'/k$ of residue fields which exists by \cite[Th{\'e}or{\`e}me 1 and Corollaire on page AC IX.40-41]{Bou06}. Let $\mathfrak{X}_A \to \Spec A$ and $\Tilde{\mathfrak{X}}_{\Tilde{A}} \to \Tilde{A}$ be the base-change of $\mathfrak{X}$ and $\Tilde{\mathfrak{X}}$, respectively. Since the extensions of dvr's are unramified, the base-changes of strictly semi-stable families are strictly semi-stable, see \Cref{sec:strictlysemistable}. Moreover, the natural morphism $q \colon \Tilde{Y} \to Y$ given by restricting $q \colon \Tilde{\mathfrak{X}} \to \mathfrak{X}$ to the special fibre induces a morphism of the base-change $q \colon \Tilde{Y}_{L'} \to Y_{L'}$. Then the same argument as in the proof above shows: If the complex \eqref{eq:complexbasechange} is exact after base change to $L'$, then the complex \eqref{eq:complexbefore} is exact after base change to $L'$.
\end{remark}

Using the same argument as in \cite[Theorem 4.1]{PS23} and the above theorem, we obtain the following result.

\begin{cor}\label{cor:mainresult2}
    Let $R$ be a dvr with algebraically closed residue field $k$, let $\mathfrak{X} \to \Spec R$ be a strictly semi-stable $R$-scheme with special fibre $Y = \bigcup\limits_{i \in I} Y_i$ such that the geometric generic fibre $\overline{X}$ of $\mathfrak{X}$ admits a decomposition of the diagonal. Then for any unramified extension $A/R$ of dvr's with induced extension of residue fields $L'/k$, the complex
    \begin{equation}\label{eq:complexmainthm}
        \bigoplus\limits_{i \in I} \CH_1\left(\left(Y_{i}\right)_{L'}\right) \xrightarrow{\Phi_{A}} \bigoplus\limits_{i \in I} \CH_0\left(\left(Y_{i}\right)_{L'}\right) \xrightarrow{\deg} \Z
    \end{equation}
    is exact, where $\Phi_{A} = \Phi_{\mathfrak{X}_A} \circ \sum_i \left(\iota_{i,L'}\right)_\ast$. Here $\Phi_{\mathfrak{X}_A}$ is the homomorphism from \Cref{def:PSmap} applied to the strictly semi-stable family $\mathfrak{X}_A \to \Spec A$ and $\iota_{i,L'} \colon \left(Y_{i}\right)_{L'} \hookrightarrow Y_{L'}$ is the natural inclusion for $i \in I$.
\end{cor}

\begin{proof}
    Since the homomorphism $\Phi_{\mathfrak{X}}$ depends only on the special fibre (see \cite[Lemma 3.2]{PS23}) and the special fibre does not change under the base-change to the completion of $R$, we can assume without loss of generality that $R$ is complete. The decomposition of the diagonal for the geometric generic fibre holds after some finite field extension $F/K$ of the fraction field of $R$. 
    
    Consider the integral closure $\Tilde{R}$ of $R$ in $F$. Since $F$ is a finite field extension and $R$ is complete, the ring $\Tilde{R}$ is also a dvr and a finite extension over $R$. Note that the residue field of $\Tilde{R}$ is again $k$ because $k$ is algebraically closed. Moreover, fix a resolution $\Tilde{\mathfrak{X}} \to \Spec \Tilde{R}$ of the base-change $\mathfrak{X}_{\Tilde{R}}$ from \Cref{prop:Hartlresolution}.

    Let $A/R$ be a unramified extension of dvr's with induced extension $L'/k$ of residue fields. Then, by \cite[Th{\'e}or{\`e}me 1 and Corollaire on page AC IX.40-41]{Bou06} there exists an unramified extension $\Tilde{A}/\Tilde{R}$ whose extension of residue fields is $L'/k$. Since the generic fibre $\Tilde{X} = X_F$ of $\Tilde{\mathfrak{X}}$ admits a decomposition of the diagonal, the homomorphism $\Phi_{\Tilde{\mathfrak{X}}_{\Tilde{A}}}$ is surjective onto the kernel of the degree map by \cite[Theorem 1.2 (1)]{PS23}. Hence, the complex \eqref{eq:complexmainthm} is exact by \Cref{thm:mainresult1}, \Cref{rem:slightgeneralization}, and \Cref{rem:nochangeincomplex}.
\end{proof}

\begin{proof}[Proof of \Cref{thm:main}]
    Let $R$ be a dvr with fraction field $k$ and let $L/k$ be a field extension. By \cite[Corollaire on page AC IX.41]{Bou06}, there exists a unramified extension $A/R$ whose induced extension of residue fields is $L/k$. Thus \Cref{thm:main} is a reformulation of \Cref{cor:mainresult2}, because the degree map $\sum_i \deg$ is clearly surjective as $k = \overline{k}$.
\end{proof}

The assumption on the residue field in \Cref{thm:main} and \Cref{cor:mainresult2} is crucial.

\begin{remark}\label{rem:counterexample}
    \Cref{thm:main} and \Cref{cor:mainresult2} show the exactness of the complex \eqref{eq:complex} for strictly semi-stable $R$-scheme if the residue field $k$ of $R$ is algebraically closed and the geometric generic fibre admits a decomposition of the diagonal. 
    
    The exactness of the complex \eqref{eq:complex} at $\Z$ requires the existence of a zero-cycle of degree $1$ on the special fibre. Such a zero-cycle does not exist in general over non-closed fields. The exactness of the complex \eqref{eq:complex} fails over non-closed fields in general also at $\bigoplus_i \CH_0(Y_i)$. Indeed, consider a smooth, projective geometrically rational variety $Y$ over a field $k \neq \overline{k}$ such that $\Azero(Y) := \ker \left(\deg \colon \CH_0(Y) \to \Z\right)$ is non-trivial. Then the smooth, projective family $\mathfrak{X} = Y \times_{k} k[[t]] \to \Spec k[[t]]$ is strictly semi-stable and the image of $\Phi$ is trivial. Hence \eqref{eq:complex} is not exact, but the geometric generic fibre is rational and thus admits a decomposition of diagonal. As an example for $Y$, we can consider the diagonal cubic surface $S_1 := \{x^3 + y^3 + z^3 + pw^3 = 0\} \subset \proj^3_{\Q_p}$ over $\Q_p$ which satisfies $\Azero(S_1) \neq 0$ by \cite[Example 2.8]{CTS96} and \cite[Theorem 4.1.1]{SS14}. By \cite[Proposition 1.3]{CT18} there is a cubic surface $S_2$ over $\R$ with $\Azero(S_2) \neq 0$, which yields another example for $Y$.
    
    We also provide a counterexample over number fields. Consider a (diagonal) cubic surface $S$ over a number field whose Brauer group $\Brauer(S)$ is non-trivial, see e.g. \cite[Proposition 1]{CTKS87} and \cite{SD93}. This implies by \cite[Theorem 2.11]{Mer08} that there exists a field extension $L/k$ such that $\Azero(S_L) \neq 0$. Hence, the complex \eqref{eq:complex} is not exact after base-change to the field extension $L/k$.
\end{remark}

\subsection{Geometrically rational varieties over Laurent fields}

We prove \Cref{cor:main} more generally for fraction fields of excellent, henselian dvr's with algebraically closed residue field. Recall that the exponential characteristic of a field $k$ is $p$ if the characteristic $\charac k = p > 0$ and $1$ if $\charac k = 0$.

\begin{cor}\label{cor:applicationCH0}
    Let $R$ be an excellent, henselian dvr with fraction field $K$ and algebraically closed residue field $k$. Let $X$ be a smooth, projective variety over $K$ and assume that $X$ admits a strictly semi-stable projective model over $R$. If $X$ is geometrically (retract) rational, then the degree map $\deg \colon \CH_0(X) \to \Z$ is an isomorphism up to inverting the exponential characteristic of $k$.
\end{cor}

\begin{proof}
    We first note that the degree map is surjective. Indeed, let $\mathfrak{X} \to \Spec R$ be a strictly semi-stable projective $R$-scheme with generic fibre $X$. Since $k$ is algebraically closed, there exists a $k$-rational point in the smooth locus of the special fibre. By Hensel's lemma we can lift this $k$-point to a section of $\mathfrak{X} \to \Spec R$, see e.g. \cite[Theorem 18.5.17]{EGAIV.4}. The intersection of this section with $X$ yields a zero-cycle of degree $1$ on $X$, i.e. the degree map is surjective.
    
    Since $X$ is geometrically retract rational, the base-change $\overline{X}$ to the algebraic closure $\Bar{K}$ admits an (integral) decomposition of the diagonal, see e.g. \cite[Lemma 7.5]{Sch21}. This decomposition of the diagonal holds after some finite field extension $F/K$. By a pull/push argument, we find that $X$ admits a rational decomposition of the diagonal. Hence the kernel of the degree map $$
        \Azero(X) := \ker \left(\CH_0(X) \to \Z \right)
    $$ is $N$-torsion for some $N \in \Z_{\geq 1}$, in fact one can choose $N = [F:k((t))]$, see e.g. \cite[proof of Lemma 1.3]{ACTP}.

    Let $Y$ denote the special fibre of the strictly semi-stable family $\mathfrak X \to \Spec R$ and let $Y_i$ denote the irreducible components of $Y$ with $i \in I$. Consider Fulton's localization exact sequence
    $$
        \CH_1(Y) \longrightarrow \CH_1(\mathfrak{X}) \longrightarrow \CH_0(X) \longrightarrow 0,
    $$
    see \cite[Section 4.4]{Ful75}. After tensoring this sequence with $\Z/l$ for some $l \in \Z_{\geq 0}$ coprime to $p = \charac k$, this sequence fits by \cite[Corollary 0.9]{SS10} into the commutative diagram
    \begin{equation}\label{eq:commdiag}
        \begin{tikzcd}
            \CH_1(Y)/l \arrow[r] \arrow[d,equal] & \CH_1(\mathfrak{X})/l \arrow[r] \arrow[d,"\cong"] & \CH_0(X)/l \arrow[d,"\deg"] \arrow[r] & 0 \\
            \CH_1(Y)/l \arrow[r,"\deg_i \circ \Phi_{\mathfrak{X}}"] & \left(\bigoplus\limits_{i \in I} \Z/l \cdot [Y_i]\right)^\vee \arrow[r,"\deg"] & \Z/l \arrow[r] & 0.
        \end{tikzcd}
    \end{equation}
    The bottom row is exact by \Cref{cor:mainresult2}. Indeed, consider the commutative diagram
    $$
        \begin{tikzcd}
            \CH_1(Y)/l \arrow[r,"\Phi_{\mathfrak{X}}"] \arrow[d,equal] & \bigoplus\limits_{i \in I}  \CH_0(Y_i)/l \arrow[r,"\deg"] \arrow[d,"\deg_i"] & \Z/l \arrow[r] \arrow[d,equal] & 0 \\
            \CH_1(Y)/l \arrow[r,"\deg_i \circ \Phi_{\mathfrak{X}}"] & \left(\bigoplus\limits_{i \in I} \Z/l \cdot [Y_i]\right)^\vee \arrow[r,"\deg"] & \Z/l \arrow[r] & 0
        \end{tikzcd}
    $$
    Since $k$ is algebraically closed, the middle arrow is surjective. The top row is exact by \Cref{cor:mainresult2}. Hence the bottom row is exact by a simple diagram chase.

    Thus, the right arrow in \eqref{eq:commdiag} is injective, in particular $\Azero(X)$ is divisible by $l$ for each $l$ coprime to $p = \charac k$. Since $\Azero(X)$ is torsion by the existence of a rational decomposition of the diagonal, the corollary follows.
\end{proof}

\begin{appendix}
\section{Two concrete resolutions}\label{appendix}
We illustrate the proof of the key proposition (\Cref{prop:key}) in two concrete examples by spelling out the constructions and arguments in \Cref{sec:CH1resolution,sec:analysis}. Recall the general setup from \Cref{setup}: We consider a strictly semi-stable family $\mathfrak X \to \Spec R$ over a dvr with special fibre $Y$ such that the irreducible components of $Y$ are geometrically integral. Let $\Tilde{\mathfrak X} \to \Spec \Tilde{R}$ be a resolution from \Cref{prop:Hartlresolution} after a finite extension $\Tilde{R}/R$ of dvr's of ramification index $r$ with induced extension $L/k$ of residue fields. We denote the special fibre of $\Tilde{\mathfrak X} \to \Spec \Tilde{R}$ by $\Tilde{Y}$ and the natural morphism by $q \colon \Tilde{Y} \to Y$.

\begin{example}\label{example:a1}
    Assume that $Y$ is a chain of three Cartier divisors and denote the three irreducible components by $Y_1,Y_2,Y_3$. The simplicial complex $\mathcal{C}_{\mathfrak X}$ associated to $\mathfrak X$ from \Cref{def:associatedcomplex} is the following
    $$
        
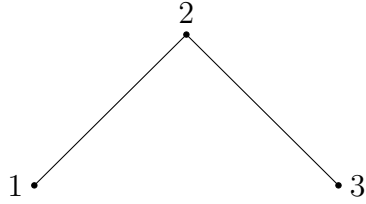
\begin{tikzpicture}
            \filldraw[black] (0,0) circle (1pt) node[anchor=east]{$1$};
            \filldraw[black] (2,2) circle (1pt) node[anchor=south]{$2$};
            \filldraw[black] (4,0) circle (1pt) node[anchor=west]{$3$};
            \draw[-] (4,0) -- (2,2);
            \draw[-] (2,2) -- (0,0);
        \end{tikzpicture}
    $$
    The refinement from \Cref{const:refinement} is then given in the case $r = 4$ by
    $$
        
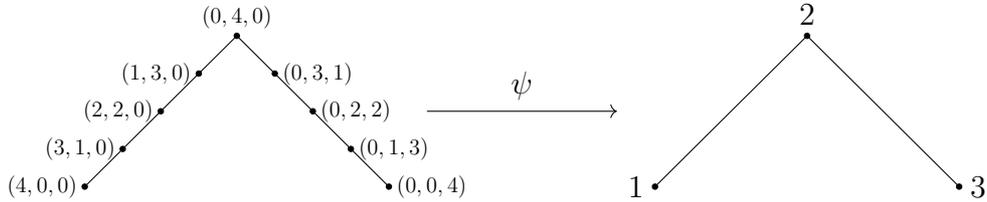
\begin{tikzpicture}
        \begin{scope}[xshift = -7.5cm,every node/.style={scale=0.7}]
            \filldraw[black] (0,0) circle (1pt) node[anchor=east]{$(4,0,0)$};
            \filldraw[black] (2,2) circle (1pt) node[anchor=south]{$(0,4,0)$};
            \filldraw[black] (4,0) circle (1pt) node[anchor=west]{$(0,0,4)$};
            \filldraw[black] (0.5,0.5) circle (1pt) node[anchor=east]{$(3,1,0)$};
            \filldraw[black] (1,1) circle (1pt) node[anchor=east]{$(2,2,0)$};
            \filldraw[black] (1.5,1.5) circle (1pt) node[anchor=east]{$(1,3,0)$};
            \filldraw[black] (2.5,1.5) circle (1pt) node[anchor=west]{$(0,3,1)$};
            \filldraw[black] (3,1) circle (1pt) node[anchor=west]{$(0,2,2)$};
            \filldraw[black] (3.5,0.5) circle (1pt) node[anchor=west]{$(0,1,3)$};
            \draw[-] (4,0) -- (2,2);
            \draw[-] (2,2) -- (0,0);
        \end{scope}
        \draw[->] (-3,1) -- node[anchor=south]{$\psi$} (-0.5,1);
        \begin{scope}
            \filldraw[black] (0,0) circle (1pt) node[anchor=east]{$1$};
            \filldraw[black] (2,2) circle (1pt) node[anchor=south]{$2$};
            \filldraw[black] (4,0) circle (1pt) node[anchor=west]{$3$};
            \draw[-] (4,0) -- (2,2);
            \draw[-] (2,2) -- (0,0);
        \end{scope}
        \end{tikzpicture}
    $$
    Recall that the refinement $\psi \colon \mathcal{C}_{\Tilde{\mathfrak{X}}} \to \mathcal{C}_{\mathfrak{X}}$ corresponds to a resolution $\Tilde{\mathfrak{X}} \to \Spec \Tilde{R}$ of $\mathfrak{X}_{\Tilde{R}} \to \Spec \Tilde{R}$. The irreducible components of the special fibre of $\Tilde{\mathfrak{X}} \to \Spec \Tilde{R}$ correspond to the vertices $(i_1,i_2,i_3) \in \N^3$ of $\mathcal{C}_{\Tilde{\mathfrak X}}$: The components $\Tilde{Y}_{(4,0,0)}$, $\Tilde{Y}_{(0,4,0)}$, and $\Tilde{Y}_{(0,0,4)}$ are isomorphic to $(Y_1)_L$, $(Y_2)_L$ and $(Y_3)_L$, respectively. The component $\Tilde{Y}_{(i,_1,i_2,i_3)}$ corresponding to the vertex $(i_1,i_2,i_3)$ with $i_2, i_1+i_3 \neq 0$ in the above picture is a (Zariski) locally trivial $\proj^1$-bundle over $(Y_j \cap Y_2)_L$ where $j \in \{1,3\}$ such that $i_j \neq 0$.

    By \Cref{prop:CH1}, we have a surjection,
    \begin{equation*}
        \begin{aligned}
            f \colon \bigoplus\limits_{j=1}^3 \CH_1(\left(Y_j\right)_L) \oplus \bigoplus\limits_{\substack{i_1,i_2 \geq 1 \\ i_1 + i_2 = 4}} \CH_0(\left(Y_1 \cap Y_2\right)_L) \oplus \bigoplus\limits_{\substack{j_2,j_3 \geq 1 \\ j_2 + j_3 = 4}} \CH_0(\left(Y_2 \cap Y_3\right)_L) \longrightarrow \CH_1(\Tilde{Y}), \\
            (\gamma_1,\gamma_2,\gamma_3,\alpha_{(i_1,i_2,0)},\alpha_{(0,j_2,j_3)}) \mapsto
            \gamma_1 + \gamma_2 +  \gamma_3 + \sum\limits_{\substack{i_1,i_2 \geq 1 \\ i_1 + i_2 = 4}}  q_{(i_1,i_2,0)}^\ast \alpha_{(i_1,i_2,0)} + \sum\limits_{\substack{j_2,j_3 \geq 1 \\ j_2 + j_3 = 4}} q_{(0,j_2,j_3)}^\ast \alpha_{(0,j_2,j_3)}
        \end{aligned}
    \end{equation*}
    where $q_{(i_1,i_2,0)} \colon \Tilde{Y}_{(i_1,i_2,0)} \to \left(Y_1 \cap Y_2\right)_L$ and $q_{(0,j_2,j_3)} \colon \Tilde{Y}_{(0,j_2,j_3)} \to \left(Y_2 \cap Y_3\right)_L$ are the natural (flat) projections of the projective bundles. For simplicity we left out the push-forwards along the natural inclusion of the irreducible components to $\Tilde{Y}$. The distance function $d$ from \Cref{const:distance} can be read of the indices of the components as follows:
    $$
        d((i_1,i_2,i_3),(4,0,0)) = 4 - i_1, \quad d((i_1,i_2,i_3),(0,4,0)) = 4 - i_2, \quad d((i_1,i_2,i_3),(0,0,4)) = 4 - i_3.  
    $$
    Note that the relative walls in this examples are the new vertices $(i_1,i_2,0)$ and $(0,j_2,j_3)$ with $i_1,\ i_2,\ j_2,\ j_3 > 0$, i.e. we denote the relative walls by the vertex. Then the function $I_{\tau}(\Tilde{v})$ from \Cref{const:intersectionnumber} is given for $\tau = (2,2,0)$ by
    $$
        \begin{aligned}
            I_{(2,2,0)}((i_1,i_2,i_3)) = \begin{cases}
                1 & \text{if } (i_1,i_2,i_3) = (1,3,0) \text{ or } (3,1,0), \\
                -2 & \text{if } (i_1,i_2,i_3) = (2,2,0), \\
                0 & \text{otherwise}
            \end{cases}
        \end{aligned}
    $$
    and similarly for the other relative walls. 
    
    The explicit description of the morphism $\Phi_{\Tilde{\mathfrak X}}$ from \Cref{lem:phiexplicit} is given as follows
    $$
        \begin{aligned}
            \Phi_{{(4,0,0)}}(\gamma) &= -\restr{\gamma_1}_{\left(Y_1 \cap Y_2\right)_L} + \alpha_{(3,1,0)} &&\in \CH_0(\Tilde{Y}_{(4,0,0)}),\\
            \Phi_{{(0,4,0)}}(\gamma) &= -\restr{\gamma_2}_{\left(Y_1 \cap Y_2\right)_L} -\restr{\gamma_2}_{\left(Y_2 \cap Y_3\right)_L} + \alpha_{(0,3,1)} + \alpha_{(1,3,0)} &&\in \CH_0(\Tilde{Y}_{(0,4,0)}),\\
            \Phi_{{(0,0,4)}}(\gamma) &= -\restr{\gamma_3}_{\left(Y_2 \cap Y_3\right)_L} + \alpha_{(0,1,3)} &&\in \CH_0(\Tilde{Y}_{(0,0,4)}), \\
            \Phi_{{(3,1,0)}}(\gamma) &= \restr{\gamma_1}_{\left(Y_1 \cap Y_2\right)_L} + \alpha_{(2,2,0)} - 2 \alpha_{(3,1,0)} &&\in \CH_0(\Tilde{Y}_{(3,1,0)}),\\
            \Phi_{{(2,2,0)}}(\gamma) &= \alpha_{(1,3,0)} + \alpha_{(3,1,0)} - 2\alpha_{(2,2,0)} &&\in \CH_0(\Tilde{Y}_{(2,2,0)}),\\
            \Phi_{{(1,3,0)}}(\gamma) &= \restr{\gamma_2}_{\left(Y_1 \cap Y_2\right)_L} + \alpha_{(2,2,0)} - 2 \alpha_{(1,3,0)} &&\in \CH_0(\Tilde{Y}_{(1,3,0)}),\\
            \Phi_{{(0,3,1)}}(\gamma) &= \restr{\gamma_2}_{\left(Y_2 \cap Y_3\right)_L} + \alpha_{(0,2,2)} - 2 \alpha_{(0,3,1)} &&\in \CH_0(\Tilde{Y}_{(0,3,1)}), \\
            \Phi_{{(0,2,2)}}(\gamma) &= \alpha_{(0,3,1)} + \alpha_{(0,1,3)} -2 \alpha_{(0,2,2)} &&\in \CH_0(\Tilde{Y}_{(0,2,2)}), \\
            \Phi_{{(0,1,3)}}(\gamma) &= \restr{\gamma_3}_{\left(Y_2 \cap Y_3\right)_L} + \alpha_{(0,2,2)} - 2 \alpha_{(0,1,3)} &&\in \CH_0(\Tilde{Y}_{(0,1,3)}),
        \end{aligned}
    $$
    where $\Phi_{(i_1,i_2,i_3)} = \Phi_{\Tilde{\mathfrak X},\Tilde{Y}_{(i_1,i_2,i_3)}}$ and $\gamma = f\left(\gamma_1,\gamma_2,\gamma_3,\alpha_{(i_1,i_2,0)},\alpha_{(0,j_2,j_3)}\right) \in \CH_1(\Tilde{Y})$. 
    
    The key lemma (\Cref{prop:key}) reads in this example as follows: 
    $$
        \begin{aligned}
            \Phi_{\mathfrak X,Y_1}(q_\ast \gamma) &= \sum\limits_{i_1 = 1}^4 i_1 \cdot q_\ast \Phi_{{(i_1,4-i_1,0)}}(\gamma), \\
            \Phi_{\mathfrak X,Y_2}(q_\ast \gamma) &= 4 q_\ast \Phi_{{(0,4,0)}}(\gamma) + \sum\limits_{i_2 = 1}^3 i_2 \cdot q_\ast \Phi_{{(4-i_2,i_2,0)}}(\gamma) + i_2 \cdot q_\ast \Phi_{{(0,i_2,4-i_2)}}(\gamma), \\
            \Phi_{\mathfrak X,Y_3}(q_\ast \gamma) &= \sum\limits_{j_3 = 1}^4 j_3 \cdot q_\ast \Phi_{{(0,4-j_3,j_3)}}(\gamma),
        \end{aligned}
    $$
    where we again leave out the push-forwards along the natural inclusions. This can be checked directly from the above description of $\Phi_{\Tilde{\mathfrak X}}$ together with the observations
    $$
        \restr{(q_\ast\gamma_i)}_{Y_i \cap Y_j} = q_\ast \left(\restr{\gamma_i}_{(Y_i \cap Y_j)_L}\right), \quad q_\ast q^\ast_{(i_1,i_2,i_3)}\alpha_{(i_1,i_2,i_3)} = 0,
    $$
    see \cite[Proposition 2.3 (c)]{Ful98} for the first claim. Note that the explicit description of $\Phi_{\mathfrak X}$ in \Cref{lem:phiexplicit} reads in this example
    $$
        \begin{aligned}
            \Phi_{\mathfrak X,Y_1}(q_\ast \gamma) &= - \restr{\left(q_\ast \gamma_1\right)}_{Y_1 \cap Y_2} + \restr{\left(q_\ast \gamma_2\right)}_{Y_1 \cap Y_2}, \\
            \Phi_{\mathfrak X,Y_2}(q_\ast \gamma) &= - \restr{\left(q_\ast \gamma_2\right)}_{Y_1 \cap Y_2} - \restr{\left(q_\ast \gamma_2\right)}_{Y_2 \cap Y_3} + \restr{\left(q_\ast \gamma_1\right)}_{Y_1 \cap Y_2} + \restr{\left(q_\ast \gamma_3\right)}_{Y_2 \cap Y_3}, \\
            \Phi_{\mathfrak X,Y_3}(q_\ast \gamma) &= - \restr{\left(q_\ast \gamma_3\right)}_{Y_2 \cap Y_3} + \restr{\left(q_\ast \gamma_2\right)}_{Y_2 \cap Y_3},
        \end{aligned}
    $$
    where we left out again the push forward along the natural inclusions.
\end{example}

\begin{example}
    Assume that $Y$ is a complete intersection of three Cartier divisors $Y_1$, $Y_2$, and $Y_3$, i.e. the simplicial complex $\mathcal{C}_{\mathfrak X}$ associated to $\mathfrak X$ is the standard $2$-simplex:
    $$
        
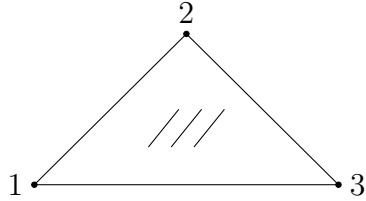
\begin{tikzpicture}
            \filldraw[black] (0,0) circle (1pt) node[anchor=east]{$1$};
            \filldraw[black] (2,2) circle (1pt) node[anchor=south]{$2$};
            \filldraw[black] (4,0) circle (1pt) node[anchor=west]{$3$};
            \draw[-] (4,0) -- (2,2);
            \draw[-] (2,2) -- (0,0);
            \draw[-] (0,0) -- (4,0);
            \draw[-] (1.5,0.5) -- (1.9,1);
            \draw[-] (1.8,0.5) -- (2.2,1);
            \draw[-] (2.1,0.5) -- (2.5,1);
        \end{tikzpicture}
    $$
    For $r = 3$, a possible refinement looks like:
    $$
        
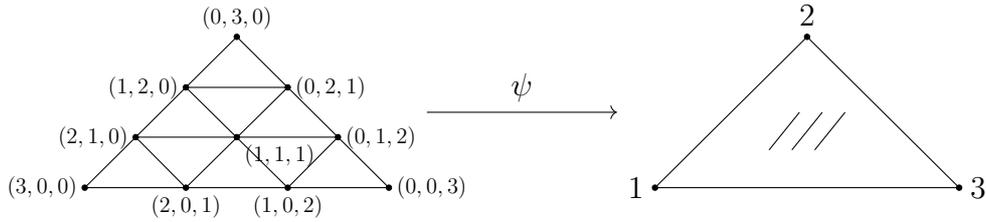
\begin{tikzpicture}
            \begin{scope}[xshift = -7.5cm,every node/.style={scale=0.7}]
            \filldraw[black] (0,0) circle (1pt) node[anchor=east]{$(3,0,0)$};
            \filldraw[black] (0.67,0.67) circle (1pt) node[anchor=east]{$(2,1,0)$};
            \filldraw[black] (1.33,1.33) circle (1pt) node[anchor=east]{$(1,2,0)$};
            \filldraw[black] (2,2) circle (1pt) node[anchor=south]{$(0,3,0)$};
            \filldraw[black] (1.33,0) circle (1pt) node[anchor=north]{$(2,0,1)$};
            \filldraw[black] (2.67,0) circle (1pt) node[anchor=north]{$(1,0,2)$};
            \filldraw[black] (4,0) circle (1pt) node[anchor=west]{$(0,0,3)$};
            \filldraw[black] (2.67,1.33) circle (1pt) node[anchor=west]{$(0,2,1)$};
            \filldraw[black] (3.33,0.67) circle (1pt) node[anchor=west]{$(0,1,2)$};
            \filldraw[black] (2,0.67) circle (1pt) node[anchor=north west]{$(1,1,1)$};
            \draw[-] (4,0) -- (2,2);
            \draw[-] (2,2) -- (0,0);
            \draw[-] (0,0) -- (4,0);
            \draw[-] (0.67,0.67) -- (3.33,0.67);
            \draw[-] (1.33,1.33) -- (2.67,1.33);
            \draw[-] (0.67,0.67) -- (1.33,0);
            \draw[-] (1.33,1.33) -- (2.67,0);
            \draw[-] (3.33,0.67) -- (2.67,0);
            \draw[-] (2.67,1.33) -- (1.33,0);
            \end{scope}
            \draw[->] (-3,1) -- node[anchor=south]{$\psi$} (-0.5,1);
            \begin{scope}
            \filldraw[black] (0,0) circle (1pt) node[anchor=east]{$1$};
            \filldraw[black] (2,2) circle (1pt) node[anchor=south]{$2$};
            \filldraw[black] (4,0) circle (1pt) node[anchor=west]{$3$};
            \draw[-] (4,0) -- (2,2);
            \draw[-] (2,2) -- (0,0);
            \draw[-] (0,0) -- (4,0);
            \draw[-] (1.5,0.5) -- (1.9,1);
            \draw[-] (1.8,0.5) -- (2.2,1);
            \draw[-] (2.1,0.5) -- (2.5,1);
        \end{scope}
        \end{tikzpicture}
    $$
    This refinement can be obtained by the blow-ups along the components corresponding to the vertices $(3,0,0),(3,0,0),(0,3,0),(1,2,0),(0,3,0),(1,2,0),(2,1,0)$ in that order. The component $\Tilde{Y}_{(1,1,1)}$ corresponding to the central vertex $(1,1,1)$ in the picture is a Zariski locally trivial $\proj^1 \times \proj^1$-bundle blown up along two disjoint sections. The components indexed by $(3,0,0),\ (0,3,0),$ and $(0,0,3)$ are isomorphic to $(Y_1)_L$, $(Y_2)_L$, and $(Y_3)_L$, respectively. The other components are blow-ups of a Zariski locally trivial $\proj^1$-bundle over the intersection $(Y_i \cap Y_j)_L$ of two components of $Y$ along the smooth locus $(Y_1 \cap Y_2 \cap Y_3)_L$ viewed via a section as a subvariety in the projective bundle. There are two types of relative walls which are depicted below in two different colors:
    $$
        
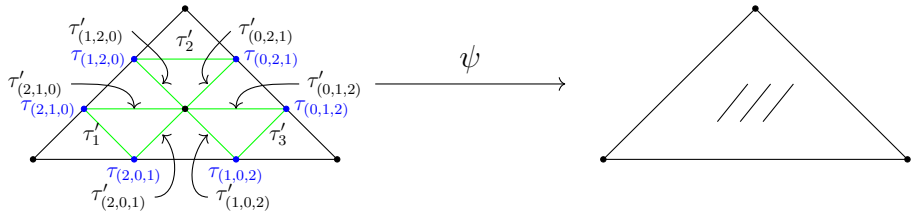
\begin{tikzpicture}
            \begin{scope}[xshift = -7.5cm,every node/.style={scale=0.7}]
            \draw[-] (4,0) -- (2,2);
            \draw[-] (2,2) -- (0,0);
            \draw[-] (0,0) -- (4,0);
            \draw[-,green] (0.67,0.67) -- node[anchor=east,black]{$\tau'_{1}$} (1.33,0);
            \draw[-,green] (1.33,1.33) -- node[anchor=south,black]{$\tau'_2$} (2.67,1.33);
            \draw[-,green] (3.33,0.67) -- node[anchor=west,black]{$\tau'_3$} (2.67,0);
            \draw[-,green] (0.67,0.67) -- (2,0.67);
            \draw[->] (0.5,1) node[anchor=east,black]{$\tau'_{(2,1,0)}$} to[out=0,in=90] (1.33,0.67);
            \draw[-,green] (1.33,1.33) -- (2,0.67);
            \draw[->] (1.3,1.7) node[anchor=east,black]{$\tau'_{(1,2,0)}$} to[out=-45,in=45] (1.7,1);
            \draw[-,green] (2,0.67) -- (1.33,0);
            \draw[->] (1.6,-0.5) node[anchor=east,black]{$\tau'_{(2,0,1)}$} to[out=0,in=-45] (1.7,0.33);
            \draw[-,green] (2,0.67) -- (2.67,0);
            \draw[->] (2.3,-0.5) node[anchor=west,black]{$\tau'_{(1,0,2)}$} to[out=-180,in=225] (2.3,0.33);
            \draw[-,green] (2.67,1.33) -- (2,0.67);
            \draw[->] (2.6,1.7) node[anchor=west,black]{$\tau'_{(0,2,1)}$} to[out=-135,in=135] (2.3,1);
            \draw[-,green] (2,0.67) -- (3.33,0.67);
            \draw[->] (3.5,1) node[anchor=west,black]{$\tau'_{(0,1,2)}$} to[out=-180,in=90] (2.67,0.67);
            \filldraw[black] (0,0) circle (1pt);
            \filldraw[blue] (0.67,0.67) circle (1pt) node[anchor=east]{$\tau_{(2,1,0)}$};
            \filldraw[blue] (1.33,1.33) circle (1pt) node[anchor=east]{$\tau_{(1,2,0)}$};
            \filldraw[black] (2,2) circle (1pt);
            \filldraw[blue] (1.33,0) circle (1pt) node[anchor=north]{$\tau_{(2,0,1)}$};
            \filldraw[blue] (2.67,0) circle (1pt) node[anchor=north]{$\tau_{(1,0,2)}$};
            \filldraw[black] (4,0) circle (1pt);
            \filldraw[blue] (2.67,1.33) circle (1pt) node[anchor=west]{$\tau_{(0,2,1)}$};
            \filldraw[blue] (3.33,0.67) circle (1pt) node[anchor=west]{$\tau_{(0,1,2)}$};
            \filldraw[black] (2,0.67) circle (1pt);
            \end{scope}
            \draw[->] (-3,1) -- node[anchor=south]{$\psi$} (-0.5,1);
            \begin{scope}
            \filldraw[black] (0,0) circle (1pt); 
            \filldraw[black] (2,2) circle (1pt);
            \filldraw[black] (4,0) circle (1pt);
            \draw[-] (4,0) -- (2,2);
            \draw[-] (2,2) -- (0,0);
            \draw[-] (0,0) -- (4,0);
            \draw[-] (1.5,0.5) -- (1.9,1);
            \draw[-] (1.8,0.5) -- (2.2,1);
            \draw[-] (2.1,0.5) -- (2.5,1);
        \end{scope}
        \end{tikzpicture}
    $$
    To simplify notation, we consider the following action of the symmetric group $S_3$ on the vertices: $\sigma(i_1,i_2,i_3) = (i_{\sigma(1)},i_{\sigma(2)},i_{\sigma(3)})$ for every $(i_1,i_2,i_3) \in \N_0^3$. By \Cref{prop:CH1}, we can write any one-cycle $\gamma \in \CH_1(\Tilde{Y})$ as
    \begin{equation*}
        \gamma = \gamma_1 + \gamma_2 + \gamma_3 + (q'_1)^\ast \alpha'_{1} + (q'_2)^\ast \alpha'_2 + (q'_3)^\ast \alpha'_3 + \sum\limits_{\sigma \in S_3} \left(q_{\sigma(2,1,0)}\right)^\ast \alpha_{\sigma(2,1,0)} + \left(q'_{\sigma(2,1,0)}\right)^\ast \alpha'_{\sigma(2,1,0)}
    \end{equation*}
    in $\CH_1(\Tilde{Y})$ for some $\gamma_i \in \CH_1((Y_i)_L)$, $\alpha_{(i_1,i_2,i_3)} \in \CH_0(\Tilde{Y}_{(i_1,i_2,i_3)})$, and $\alpha_1,\alpha_2,\alpha_3, \alpha'_{(i_1,i_2,i_3)} \in \CH_0((Y_1 \cap Y_2 \cap Y_3)_L)$, where the zero-cycles $\alpha_{\dots}$ and $\alpha'_{\dots}$ correspond to the respective relative wall in the picture above and the morphisms $q_{\dots}$ and $q'_{\dots}$ are the natural projection of the associated $\proj^1$-bundle. Note that we leave out the push-forwards along the respective natural inclusions.

    As in \Cref{example:a1}, the distance function $d$ from \Cref{const:distance} can be read of the indices of the components as follows:
    $$
        d((i_1,i_2,i_3),(4,0,0)) = 4 - i_1, \quad d((i_1,i_2,i_3),(0,4,0)) = 4 - i_2, \quad d((i_1,i_2,i_3),(0,0,4)) = 4 - i_3.  
    $$
    We write down the function $I$ representing the intersection numbers by the matrix
    $$
        \left(\begin{array}{*{15}c}
            1 & 0 & 1 & 0 & 0 & 0 & 1 & 0 & 0 & 0 & 0 & 0 & 0 & 0 & 0 \\
            0 & 1 & 0 & 0 & 1 & 0 & 0 & 1 & 0 & 0 & 0 & 0 & 0 & 0 & 0 \\
            0 & 0 & 0 & 1 & 0 & 1 & 0 & 0 & 1 & 0 & 0 & 0 & 0 & 0 & 0 \\
            -2 & 1 & 0 & 0 & 0 & 0 & -1 & 0 & 0 & -1 & 1 & 1 & 0 & 0 & 0 \\
            1 & -2 & 0 & 0 & 0 & 0 & 0 & -1 & 0 & 1 & -1 & 0 & 0 & 1 & 0 \\
            0 & 0 & -2 & 1 & 0 & 0 & -1 & 0 & 0 & 1 & 0 & -1 & 1 & 0 & 0 \\
            0 & 0 & 1 & -2 & 0 & 0 & 0 & 0 & -1 & 0 & 0 & 1 & -1 & 0 & 1 \\
            0 & 0 & 0 & 0 & -2 & 1 & 0 & -1 & 0 & 0 & 1 & 0 & 0 & -1 & 1 \\
            0 & 0 & 0 & 0 & 1 & -2 & 0 & 0 & -1 & 0 & 0 & 0 & 1 & 1 & -1 \\
            0 & 0 & 0 & 0 & 0 & 0 & 1 & 1 & 1 & -1 & -1 & -1 & -1 & -1 & -1
        \end{array}\right)
    $$
    where the columns are labelled in the following order:
    $$
        \tau_{(2,1,0)}, \tau_{(1,2,0)}, \tau_{(2,0,1)}, \tau_{(1,0,2)}, \tau_{(0,2,1)}, \tau_{(0,1,2)}, \tau'_1, \tau'_2, \tau'_3, \tau'_{(2,1,0)}, \tau'_{(1,2,0)}, \tau'_{(2,0,1)}, \tau'_{(1,0,2)}, \tau'_{(0,2,1)}, \tau'_{(0,1,2)},
    $$
    and the rows represent the component $\Tilde{Y}_{(i_1,i_2,i_3)}$ in the following order
    $$
        (3,0,0), (0,3,0), (0,0,3), (2,1,0), (1,2,0), (2,0,1), (1,0,2), (0,2,1), (0,1,2), (1,1,1). 
    $$
    For simplicity, we denote the map $\Phi_{\Tilde{\mathfrak{X}},\Tilde{Y}_{(i_1,i_2,i_3)}}$ by $\Phi_{(i_1,i_2,i_3)}$. Then, for $\gamma \in \CH_1(\Tilde{Y})$ of the above form the map $\Phi_{\Tilde{\mathfrak X}}$ is given
    $$
        \begin{aligned}
            \Phi_{(3,0,0)}(\gamma) &= -\restr{\gamma_1}_{(Y_1 \cap Y_2)_L} - \restr{\gamma_1}_{(Y_1 \cap Y_3)_L} + \alpha_{(2,1,0)} + \alpha_{(2,0,1)} + \alpha'_1 &&\in \CH_0(\Tilde{Y}_{(3,0,0)}), \\
            \Phi_{(0,3,0)}(\gamma) &= -\restr{\gamma_2}_{(Y_1 \cap Y_2)_L} - \restr{\gamma_2}_{(Y_2 \cap Y_3)_L} + \alpha_{(1,2,0)} + \alpha_{(0,2,1)} + \alpha_2' &&\in \CH_0(\Tilde{Y}_{(0,3,0)}), \\
            \Phi_{(0,0,3)}(\gamma) &= -\restr{\gamma_3}_{(Y_1 \cap Y_3)_L} - \restr{\gamma_3}_{(Y_2 \cap Y_3)_L} + \alpha_{(1,0,2)} + \alpha_{(0,1,2)} + \alpha_3' &&\in \CH_0(\Tilde{Y}_{(0,0,3)}), \\
            \Phi_{(2,1,0)}(\gamma) &= \restr{\gamma_1}_{(Y_1 \cap Y_2)_L} -2 \alpha_{(2,1,0)} + \alpha_{(1,2,0)} - \alpha'_1 - \alpha'_{(2,1,0)} + \alpha'_{(1,2,0)} + \alpha'_{(2,0,1)} &&\in \CH_0(\Tilde{Y}_{(2,1,0)}), \\
            \Phi_{(1,2,0)}(\gamma) &= \restr{\gamma_2}_{(Y_1 \cap Y_2)_L} + \alpha_{(2,1,0)} - 2\alpha_{(1,2,0)} - \alpha'_2 + \alpha'_{(2,1,0)} - \alpha'_{(1,2,0)} + \alpha'_{(0,2,1)} &&\in \CH_0(\Tilde{Y}_{(1,2,0)}), \\
            \Phi_{(2,0,1)}(\gamma) &= \restr{\gamma_1}_{(Y_1 \cap Y_3)_L} -2\alpha_{(2,0,1)} + \alpha_{(1,0,2)} - \alpha'_1 + \alpha'_{(2,1,0)} - \alpha'_{(2,0,1)} + \alpha'_{(1,0,2)} &&\in \CH_0(\Tilde{Y}_{(2,0,1)}), \\
            \Phi_{(1,0,2)}(\gamma) &= \restr{\gamma_3}_{(Y_1 \cap Y_3)_L} + \alpha_{(2,0,1)} -2\alpha_{(1,0,2)} - \alpha'_3 + \alpha'_{(2,0,1)} - \alpha'_{(1,0,2)} + \alpha'_{(0,1,2)} &&\in \CH_0(\Tilde{Y}_{(1,0,2)}), \\
            \Phi_{(0,2,1)}(\gamma) &= \restr{\gamma_2}_{(Y_2 \cap Y_3)_L} - 2\alpha_{(0,2,1)} + \alpha_{(0,1,2)} - \alpha'_2 + \alpha'_{(1,2,0)} - \alpha'_{(0,2,1)} + \alpha'_{(0,1,2)} &&\in \CH_0(\Tilde{Y}_{(0,2,1)}), \\
            \Phi_{(0,1,2)}(\gamma) &= \restr{\gamma_3}_{(Y_2 \cap Y_3)_L} + \alpha_{(0,2,1)} - 2\alpha_{(0,1,2)} - \alpha'_3 + \alpha'_{(1,0,2)} + \alpha'_{(0,2,1)} - \alpha'_{(0,1,2)} &&\in \CH_0(\Tilde{Y}_{(0,1,2)}), \\
            \Phi_{(1,1,1)}(\gamma) &= \alpha'_1 + \alpha'_2 + \alpha'_3 - \sum\limits_{\sigma \in S_3} \alpha'_{\sigma(2,1,0)} &&\in \CH_0(\Tilde{Y}_{(1,1,1)}), \\
        \end{aligned}
    $$
    where we left out the push forwards along the natural inclusions. Note that $\CH_0(\Tilde{Y}_{(2,1,0)}) = \CH_0((Y_1 \cap Y_2)_L)$ as $\Tilde{Y}_{(2,1,0)}$ is a blow-up in smooth centers of a locally trivial $\proj^1$-bundle over $(Y_1 \cap Y_2)_L$ (and similarly for the other components), i.e. we view the zero-cyles $\alpha_{(1,2,0)}$ as a zero-cycle on $\Tilde{Y}_{(2,1,0)}$. The key lemma reads
    $$
        \begin{aligned}
            \sum\limits_{\substack{i_1,i_2,i_3 \\ i_1 > 0}} i_1 q_\ast \Phi_{(i_1,i_2,i_3)}(\gamma) &= - \restr{\left(q_\ast \gamma_1\right)}_{Y_1 \cap Y_2} - \restr{\left(q_\ast \gamma_1\right)}_{Y_1 \cap Y_3} + \restr{\left(q_\ast \gamma_2\right)}_{Y_1 \cap Y_2} + \restr{\left(q_\ast \gamma_3\right)}_{Y_1 \cap Y_3} = \Phi_{\mathfrak X,Y_1}(q_\ast \gamma), \\
            \sum\limits_{\substack{i_1,i_2,i_3 \\ i_2 > 0}} i_2 q_\ast \Phi_{(i_1,i_2,i_3)}(\gamma) &= - \restr{\left(q_\ast \gamma_2\right)}_{Y_1 \cap Y_2} - \restr{\left(q_\ast \gamma_2\right)}_{Y_2 \cap Y_3} + \restr{\left(q_\ast \gamma_1\right)}_{Y_1 \cap Y_2} + \restr{\left(q_\ast \gamma_3\right)}_{Y_2 \cap Y_3} = \Phi_{\mathfrak X,Y_2}(q_\ast \gamma), \\
            \sum\limits_{\substack{i_1,i_2,i_3 \\ i_3 > 0}} i_3 q_\ast \Phi_{(i_1,i_2,i_3)}(\gamma) &= - \restr{\left(q_\ast \gamma_3\right)}_{Y_1 \cap Y_3} - \restr{\left(q_\ast \gamma_3\right)}_{Y_2 \cap Y_3} + \restr{\left(q_\ast \gamma_1\right)}_{Y_1 \cap Y_3} + \restr{\left(q_\ast \gamma_2\right)}_{Y_2 \cap Y_3} = \Phi_{\mathfrak X,Y_3}(q_\ast \gamma),
        \end{aligned}
    $$
    which can be checked immediately by the above description together with the observations:
    $$
        q_\ast \gamma = q_\ast \gamma_1 + q_\ast \gamma_2 + q_\ast \gamma_3, \text{ and } \restr{(q_\ast(\gamma_i))}_{Y_i \cap Y_j} = q_\ast \left(\restr{\gamma_i}_{(Y_i \cap Y_j)_L}\right),
    $$
    where the latter is again \cite[Proposition 2.3 (c)]{Ful98}.
\end{example}
\end{appendix}

\printbibliography
\end{document}